\documentclass[11pt,twoside,a4paper]{article}
\usepackage{amsmath}
\usepackage{amssymb}
\usepackage{amsthm}
\usepackage{gensymb, textcomp}

\usepackage{euscript}
\usepackage{color}

\usepackage{mathrsfs}

\usepackage{mathrsfs}

\usepackage{libertine}
\usepackage[T1]{fontenc}

\definecolor{light}{gray}{0.8}

\newtheorem{theorem}{Theorem}[section]
\newtheorem{lemma}[theorem]{Lemma}
\newtheorem{proposition}[theorem]{Proposition}

\newtheorem{Ex}[theorem]{Example}

\newtheorem{definition}[theorem]{Definition}
\newtheorem{remark}[theorem]{Remark}
\newtheorem{corollary}[theorem]{Corollary}

\renewcommand{\vec}[1]{\boldsymbol{#1}}

\newcommand{\vb}{\vspace{3mm}}

\newcommand{\hN}{^{[N]}}

\newcommand{\DD}{{\rm d}}

\newcommand{\PP}{{\mathbb P}}
\newcommand{\EE}{{\mathbb E}}

\newcommand{\vt}{\vartheta}

\newcommand{\diag}{\mathrm{diag}}

\newcommand{\ii}{\mathrm{i}}
\newcommand{\rr}{\mathbb{R}}
\newcommand{\dd}{\mathrm{d}}
\newcommand{\ee}{\mathbb{E}}
\newcommand{\half}{\frac{1}{2}}
\newcommand{\cov}{{\mathbb C}{\rm ov}}
\newcommand{\one}{\mathbf{1}}

\newcommand{\N}{{N}}

\newcommand{\m}{\boldsymbol}

\newcommand{\sN}{\hN}
\newcommand{\tN}{^{[N,h]}}
\newcommand{\mar}{\Phi}

\newcommand{\bz}{{\m Z}}
\newcommand{\bpi}{{\boldsymbol \pi}}
\newcommand{\rmt}{\mathrm{T}}
\newcommand{\bzetan}{{\boldsymbol \zeta}^{[N]}}

\newcommand{\ynh}{Y^{[N,h]}}
\newcommand{\balpha}{\boldsymbol\alpha}
\newcommand{\bgamma}{\boldsymbol\gamma}
\newcommand{\bsigma}{\boldsymbol\sigma}

\let\originalleft\left
\let\originalright\right
\renewcommand{\left}{\mathopen{}\mathclose\bgroup\originalleft}
\renewcommand{\right}{\aftergroup\egroup\originalright}

\begin{document}
\title{{Markov-modulated Ornstein-Uhlenbeck processes}}

\author{G.\ Huang$^\bullet$, H.\ M. Jansen$^{\bullet,\dagger}$, M.\ Mandjes$^{\bullet,\star}$, P.\ Spreij$^\bullet$, \& K.\ De Turck$^\dagger$}
\maketitle
\begin{abstract} \noindent
In this paper we consider an Ornstein-Uhlenbeck ({\sc ou}) process $(M(t))_{t\geqslant 0}$ whose 
parameters are determined by an external  Markov process $(X(t))_{t\geqslant 0}$ on a finite state space $\{1,\ldots,d\}$; 
this process is usually referred to as {\it Markov-modulated Ornstein-Uhlenbeck} (or: {\sc mmou}). We use stochastic integration theory 
to determine explicit expressions for the mean and variance of $M(t)$. Then we establish a system of partial differential equations ({\sc pde}\,s) for the Laplace transform of $M(t)$ and the state $X(t)$ of the background process, jointly for time epochs $t=t_1,\ldots,t_K.$ Then we use this {\sc pde} to set up a recursion that yields all moments of $M(t)$ and its stationary counterpart; we also
find an expression for the covariance between $M(t)$ and $M(t+u)$.
We then establish a functional central limit theorem for $M(t)$ for the situation that certain parameters of  the underlying {\sc ou} processes are scaled, in combination with the modulating Markov process being accelerated; interestingly, specific scalings lead to drastically different limiting processes.
We conclude the paper by considering the situation of a single Markov process modulating {\it multiple} {\sc ou} processes.

\vb

\noindent {\it Keywords.} Ornstein-Uhlenbeck processes $\star$ Markov modulation $\star$ regime
switching $\star$ martingale techniques $\star$ central-limit theorems

\vb

\noindent {Work partially done while K.\ de Turck was visiting Korteweg-de Vries Institute for Mathematics,
University of Amsterdam, the Netherlands, with greatly appreciated financial support from {\it Fonds Wetenschappelijk Onderzoek / Research Foundation -- Flanders}.

\begin{itemize}
\item[$^\bullet$] Korteweg-de Vries Institute for Mathematics,
University of Amsterdam, Science Park 904, 1098 XH Amsterdam, the Netherlands.
\item[$^\star$] CWI, P.O. Box 94079, 1090 GB Amsterdam, the Netherlands.
\item[$^\dagger$] TELIN, Ghent University, St.-Pietersnieuwstraat 41,
B9000 Gent, Belgium.
\end{itemize}
\noindent
M.\ Mandjes is also with  E{\sc urandom}, Eindhoven University of Technology, Eindhoven, the Netherlands,
and IBIS, Faculty of Economics and Business, University of Amsterdam,
Amsterdam, the Netherlands.}

\vb

\noindent {\it Email}. {\scriptsize $\{$\tt{g.huang|h.m.jansen|m.r.h.mandjes|p.j.c.spreij}$\}$\tt{@uva.nl}, 
 {\tt kdeturck@telin.ugent.be}}
\end{abstract}

\newpage

\section{Introduction}
The Ornstein-Uhlenbeck ({\sc ou}) process is a stationary Markov-Gauss process, with the additional feature that is eventually reverts to its long-term mean; see the seminal paper \cite{OU}, as well as \cite{JAC}
for a historic account. Having originated from physics,  by now the process has found widespread use in a broad range of other application domains: finance, population dynamics, climate modeling, etc. In addition, it plays an important role in queueing theory, as it can be seen as the limiting process of specific classes of infinite-server queues under a certain scaling \cite{ROBE}.
The {\sc ou} process is characterized by three parameters (which we call $\alpha$, $\gamma$, and $\sigma^2$ throughout this paper), which relate to the process' mean, convergence speed towards the mean, and variance, respectively.

The probabilistic properties of the {\sc ou} process have been thoroughly studied. One of the key results is that
its value at a given time $t$ has a Normal distribution, with a  mean and variance that can be expressed explicitly in terms of the parameters $\alpha$, $\gamma$, and $\sigma^2$ of the underlying {\sc ou} process; see for instance \cite[Eqn.\ (2)]{JAC}. In addition, various other quantities have been analyzed, such as the distribution of first passage times  or the maximum value attained in an interval of given length; see e.g.\ \cite{ALI} and references therein.

The concept of {\it regime switching} (or:\ {\it Markov modulation}, as it is usually referred to in the operations research literature) has become increasingly important over the past decades.  In regime switching, the parameters of the underlying stochastic process are determined by an external {\it background process} (or: {\it modulating process}), that is typically assumed to evolve independently of the stochastic process under consideration. Often the background process is assumed to be a Markov chain defined on a {\it finite} state space, say $\{1,\ldots,d\}$; in the context of Markov-modulated {\sc ou} ({\sc mmou}) this means that when this Markov chain is in state $i$, the process locally behaves as a {\sc ou} process with parameters $\alpha_i$, $\gamma_i$, and $\sigma^2_i.$

\vb

Owing to its various attractive features, regime switching has become an increasingly popular concept. In a broad spectrum of application domains it offers a  natural framework for modeling 
situations in which the stochastic process under study reacts to 
an autonomously evolving environment. In finance, for instance, one could identify the background process with the `state of the economy', for instance as a two-state process (that is, alternating between a `good' and a `bad' state), to which e.g.\ asset prices react. Likewise,  in wireless networks the concept can be used to model the channel conditions that vary in time, and to which users react. 

In the operations research literature there is a sizable body of work on Markov-modulated queues, see e.g.\ the textbooks \cite[Ch.\ XI]{ASM} and \cite{NEUTS}, while Markov modulation has been intensively used in insurance and risk theory as well \cite{AA}. In the financial economics literature, the use of regime switching dates back to at least the late 1980s \cite{HAM}; various specific models have been considered since then, see for instance \cite{ANG,EMAM,ESIU}.

\vb

In this paper we present a set of new results in the context of the analysis of {\sc mmou}. Here and in the sequel we let $M(t)$ denote the position of the {\sc mmou} process at time $t$, whereas $M$ denotes its stationary counterpart.
In the first place we derive explicit formulas for the mean and variance of $M(t)$ and $M$, jointly with the state of the background process, relying on standard machinery from stochastic integration theory. In specific special cases the resulting formulas simplify drastically (for instance when it is assumed that the background process starts off in equilibrium at time 0, or when the parameters $\gamma_i$ are assumed uniform across the states $i\in\{1,\ldots,d\}$). 

The second contribution concerns the derivation of a system of partial differential equations for the Laplace transform of $M(t)$; when equating the partial derivative with respect to time to 0, we obtain a system of ordinary differential equations for the Laplace transform of $M$. This result is directly related to \cite[Thm. 3.2]{XING}, with the differences being that there the focus is on just stationary behavior, and that the system considered there has the additional feature of reflection at a lower boundary (to avoid the process attaining negative values). We set up a recursive procedure that generates all moments of $M(t)$; in each iteration a non-homogeneous system of differential equations needs to be solved. This procedure complements the recursion for the moments of the steady-state
quantity  $M$, as presented in \cite[Corollary 3.1]{XING} (in which each recursion step amounts to solving a system of linear equations). In addition, we also set up a system of partial differential equations for the Laplace transform associated with the joint distribution of $M(t_1),\ldots, M(t_K)$, and determine the covariance ${\mathbb C}{\rm ov}\,(M(t,t+u)).$

A third contribution concerns the behavior of the {\sc mmou} process under certain parameter scalings. 
\begin{itemize}
\item[$\rhd$]
A first scaling that we consider concerns speeding up the jumps of the background process by a factor $N$. Using the system of partial differential equations that we derived earlier, it is shown that the limiting process, obtained by sending $N\to\infty$, is an {\it ordinary} (that is, non-modulated) {\sc ou} process, with parameters that are time averages of the individual $\alpha_i$, $\gamma_i$, and $\sigma_i^2$.  
\item[$\rhd$]
A second  regime that we consider scales the transition rates of the Markovian background process by $N$, while the $\alpha_i$ and $\sigma_i^2$ are inflated by a factor $N^h$ for some $h>0$; the resulting process we call $M\tN(t)$. We then center (subtract the mean, which is roughly proportional to $N^h$) and normalize $M\tN(t)$, with the goal  to establish a central limit theorem ({\sc clt}). Interestingly, it depends on the value of $h$ 
what the appropriate normalization is. If $h<1$ the variance of $M\tN(t)$ is roughly proportional to the `scale' at which the modulated {\sc ou} process operates, viz.\ $N^h$, and as a consequence the normalization looks like $N^{h/2}$; at an intuitive level, the timescale of the background process is so fast, that the process essentially looks like an {\sc ou} process with time-averaged parameters. If, on the contrary, $h>1$, then the variance of $M\tN(t)$ grows like $N^{2h-1}$, which is faster than $N^h$; as a consequence, the proper normalization looks like $N^{h-1/2}$; in this case the variance that appears in the {\sc clt} is directly related to the deviation matrix \cite{CS} associated with the background process.
Importantly, we do not just prove Normality for a given value of $t>0$, but rather weak convergence (at the process level, that is) to the solution of a specific limiting stochastic differential equation.

\end{itemize}
The last contribution focuses on the situation that a single Markovian background process modulates {\it multiple} {\sc ou} processes. This, for instance, models the situation in which different asset prices react to the same `external circumstances' (i.e., state of the economy), or the situation in which different users of a wireless network react to the same channel conditions. The probabilistic behavior of the system is captured through a system of partial differential equations. It is also pointed out how the corresponding moments can be found.

\vb

Importantly, there is a strong similarity between the results presented in the framework of the present paper, and corresponding results for Markov-modulated {\it infinite-server queues}. In these systems the background process modulates an M/M/$\infty$ queue, meaning that we consider an M/M/$\infty$ queue of which  the arrival rate and service rate are determined by the state of the background process \cite{DAURIA,FRALIXADAN2009}. For these systems, the counterparts of our {\sc mmou} results have been established: the mean and variance have been computed in e.g.\ \cite{BKMT,OCP}, (partial) differential equations
for the Laplace transform of $M(t)$, as well as recursions for higher moments can be found in \cite{BMT,BKMT,OCP}, whereas parameter scaling results are given in \cite{BMT,BKMT} and, for a slightly different model \cite{BMTh}. Roughly speaking, any property that can be handled explicitly for the Markov-modulated infinite-server queue can be explicitly addressed for {\sc mmou} as well, and vice versa.

\vb

This paper is organized as follows. Section \ref{MOD} defines the model, and presents preliminary results. Then Section \ref{TB} deals with the system's transient behavior, in terms of a recursive scheme that yields all moments of $M(t)$, with explicit expressions for the mean and variance.
Section \ref{TBPDE} presents a system of partial differential equations for the Laplace transform of $M(t)$ (which becomes a system of ordinary differential equations in steady state). In Section \ref{Sec:TS}, the parameter scalings mentioned above are applied (resulting in a process $M^{(N)}(t)$), leading to a functional {\sc clt} for an appropriately centered and normalized version of $M^{(N)}(t)$. The last section considers the setting of a single background process modulating multiple {\sc ou} processes.

\section{Model and preliminaries}\label{MOD}
We start by giving a detailed model description of the Markov-modulated Ornstein-Uhlenbeck ({\sc mmou}) process. We are given a probability space $\left( \Omega , \mathcal{F} , \mathbb{P} \right)$ on which a random variable $M_0$, a standard Brownian motion $(B(t))_{t \geqslant 0}$ and a continuous-time Markov process $(X(t))_{t \geqslant 0}$ with finite state space are defined. It is assumed that $M_0$, $X$ and $B$ are independent. The process $X$ is the so-called background process; its state space is denoted by $\left\lbrace 1 , \ldots , d \right\rbrace$.

The idea behind {\sc mmou} is that the background process $X \left( \cdot \right)$ modulates an Ornstein-Uhlenbeck process. Intuitively, this means that while $X \left( \cdot \right)$ is in state $i \in \left\lbrace 1 , \ldots , d \right\rbrace$, the {\sc mmou} process $\left( M \left( t \right) \right)_{t \geqslant 0}$ behaves as an Ornstein-Uhlenbeck process $U_{i} \left( \cdot \right)$ with parameters $\alpha_{i}$, $\gamma_{i}$ and $\sigma_{i}$, which evolves independently of the background process $X \left( \cdot \right)$. In mathematical terms, this means that $M \left( \cdot \right)$ should obey the stochastic differential equation
\begin{align}
\label{SDE:MMOU}
{\rm d} M \left( t \right) = \left( \alpha_{X \left( t \right)} - \gamma_{X \left( t \right)} M \left( t \right) \right) \, {\rm d} t + \sigma_{X \left( t \right)} \, {\rm d} B \left( t \right).
\end{align}

To be more precise, we will call a stochastic process $(M(t))_{t \geqslant 0}$ an {\sc mmou} process with initial condition $M(0) = M_0$ if
\begin{align}
\label{SIE:intMMOU}
M(t) = M_0 + \int_{0}^{t} \left( \alpha_{X \left( s \right)} - \gamma_{X \left( s \right)} M \left( s \right) \right) \, {\rm d} s + \int_{0}^{t} \sigma_{X \left( s \right)} \, {\rm d} B \left( s \right).
\end{align}
The following theorem provides basic facts about the existence, uniqueness and distribution of an {\sc mmou} process. For proofs and additional details, see Section~\ref{sec:mmouexist}. As mentioned in the introduction, specific aspects of {\sc mmou} have been studied earlier in the literature; see for instance \cite{XING}.

\begin{theorem}\label{PROP}
Define $\Gamma (t) := \int_{0}^{t} \gamma_{X (s)} \, {\rm d} s$. Then the stochastic process $(M(t))_{t \geqslant 0}$ given by
\begin{align*}
M \left( t \right) = M_0 e^{ - \Gamma \left( t \right) } + \int_{0}^{t} e^{ - \left( \Gamma \left( t \right) - \Gamma \left( s \right) \right) } \alpha_{X \left( s \right)} \, {\rm d} s 
+ \int_{0}^{t} e^{ - \left( \Gamma \left( t \right) - \Gamma \left( s \right) \right) } \sigma_{X \left( s \right)} \, {\rm d} B \left( s \right)
\end{align*}
is the unique {\sc mmou} process with initial condition $M_0$. 

Conditional on the process $X$, the random variable $M \left( t \right)$ has a Normal distribution with random mean
\begin{equation}
\label{rmu}
\mu \left( t \right) = M _0 \exp \left( - \Gamma \left( t \right) \right) + \int_{0}^{t} \exp \left( - \left( \Gamma \left( t \right) - \Gamma \left( s \right) \right) \right) \alpha_{X \left( s \right)} \, {\rm d} s\end{equation}
and random variance
\begin{equation}\label{rv}
v \left( t \right) = \int_{0}^{t} \exp \left( -2 \left( \Gamma \left( t \right) - \Gamma \left( s \right) \right) \right) \sigma_{X \left( s \right)}^2 \, {\mathrm d} s.
\end{equation}
\end{theorem}

This result is analogous with the corresponding result for the Markov-modulated infinite-server queue in \cite{BKMT,DAURIA}: there it is shown that the number of jobs in the system has a Poisson distribution with random parameter.

\vb

For later use, we now recall some concepts pertaining to the theory of deviation matrices of Markov processes.
For an introduction to this topic we refer to standard texts such as \cite{KEIL,KEM,SYS}. For a compact survey, see \cite{CS}.

Let the transition rates corresponding to the continuous-time Markov chain $(X(t))_{t\geqslant 0}$ be given by $q_{ij}\geqslant 0$ for $i\not=j$ and $q_i:=-q_{ii}:=\sum_{j\not=i}q_{ij}$. These transition rates define the {\it intensity matrix} or {\it generator} $Q$. The (unique) invariant distribution corresponding to $Q$ is denoted by (the column vector) ${\boldsymbol \pi}$, i.e.,
it obeys ${\boldsymbol \pi}Q = {\boldsymbol 0}$ and $ \vec{1}^{\rm T} \vec{\pi}=1$, where
$ \vec{1}$ is a $d$-dimensional all-ones vector.

Let $\Pi := \vec{1} \vec{\pi}^{\rm T}$ denote the {\it ergodic matrix.}
Then the {\it fundamental matrix} is given by $F := (\Pi - Q)^{-1}$, whereas the {\it deviation matrix} is defined by $D := F - \Pi$.
Standard identities are $QF=FQ=\Pi - I$, as well as $\Pi D= D \Pi =0$ (here $0$ is to be read as an all-zeros $d\times d$ matrix)
and $F\vec{1}=\vec{1}.$ The $(i,j)$-th entry of the deviation matrix,
with $i,j\in\{1,\ldots,d\}$, can be alternatively computed as
\[
D_{ij}:= \int_0^\infty({\mathbb P}(X(t)=j\,|\,X(0)=i)-\pi_j){\rm d}t,\]
which in matrix form reads
\begin{equation}\label{eq:d}
D=\int_0^\infty \left(\exp(Qt)-\one\vec{\pi}^{\rm T}\right)\,\dd t.
\end{equation}

\section{Transient behavior: moments}\label{TB}
In this section we analyze the moments of $M(t)$. First considering the mean and variance in the general situation, we then  concentrate on more specific cases 
in which the expressions simplify greatly. In particular, we address the situation that all the $\gamma_i$s are equal, the situation that the background process starts off in equilibrium at time $0$, and the steady-state regime.
The computations are immediate applications of stochastic integration theory.
The section is completed by deriving an expression for the covariance between $M(t)$ and $M(t+u)$
(for $t,u\geqslant 0$), and a procedure that uses It\^o's formula to recursively determine all moments. 

\subsection{Mean and variance: general case}\label{section:gc}
Let ${\boldsymbol Z}(t)\in\{0,1\}^d$ be the vector of indicator functions associated with the Markov chain $(X(t))_{t\geqslant 0}$, that is, we let $Z_i(t)=1$ if $X(t)=i$ and $0$ else.
Let $\vec{p}_t$ denote the vector of transient probabilities of the background process, i.e., $({\mathbb P}(X(t) = 1),\ldots,{\mathbb P}(X(t)=d))^{\rm T}$
(where we have not specified the distribution of the initial state $X(0)$ yet).

We subsequently find expressions for the mean $\mu_t:={\mathbb E}M(t)$ and variance $v_t:={\mathbb V}{\rm ar}\,M(t)$.
\begin{itemize}
\item[$\rhd$]
The mean can be computed as follows. 
To this end, we consider the mean of $M(t)$ jointly with the state of the background process at time $t$.
To this end, we define $\vec{Y}(t):=\vec{Z}(t) M(t),$ and $\vec{\nu}_t:={\mathbb E}\vec{Y}(t).$
It is clear that 
\begin{equation}\label{eq:zz}
{\rm d}\vec{Z}(t) = Q^{\rm T}\vec{Z}(t)\,{\rm d}t +{\rm d} \vec{K}(t),
\end{equation} 
for a $d$-dimensional martingale $\vec{K}(t).$ With It\^o's rule we get, with $\bar{Q}_{\vec{\gamma}}:=Q^{\rm T}-{\rm diag}\{\vec{\gamma}\}$,
\begin{align}
{\rm d}\vec{Y}(t) & =M(t)\left(Q^{\rm T}\vec{Z}(t)\,{\rm d}t+{\rm d}\vec{K}(t)\right)+\vec{Z}(t)\left(\left(\vec{\alpha}^{\rm T}\vec{Z}(t)-\vec{\gamma}^{\rm T}\vec{Y}(t)\right){\rm d}t+\vec{\sigma}^{\rm T}\vec{Z}(t){\rm d}B(t)\right) \nonumber\\
& = \left(\bar{Q}_{\vec{\gamma}} \vec{Y}(t)+\diag(\alpha)\vec{Z}(t)
\right)\,\dd t +\diag\{\vec{\sigma}\}\vec{Z}(t){\rm d}B(t)+M(t)\,\dd \vec{K}(t).\label{eq:y}
\end{align}
Taking expectations of both sides, we obtain 
the system
\[\vec{\nu}'_t = \bar{Q}_{\vec{\gamma}}\vec{\nu}_t +{\rm diag}\{\vec{\alpha}\}\vec{p}_t.\]
This is a non-homogeneous linear system of differential equations, that is solved by\[\vec{\nu}_t=e^{\bar{Q}_{\vec{\gamma}}t}\vec{\nu}_0+\int_0^t e^{\bar Q_{\vec{\gamma}}(t-s)}{\rm diag}\{\vec{\alpha}\}\vec{p}_s{\rm d}s;\] then $\mu_t=\vec{1}^{\rm T}{\vec{\nu}_t}.$ Realize that $\vec{\nu}_0 =m_0\vec{p}_0$, as we assumed that $M(0)$ equals $m_0.$

The equations simplify drastically if the background process starts off in equilibrium at time $0$; then evidently $\vec{p}_t=\vec{\pi}$ for all $t\ge 0$. As a result,
we find $\vec{\nu}_t=e^{\bar{Q}_{\vec{\gamma}}t}\vec{\nu}_0-\bar{Q}_{\vec{\gamma}}^{-1}(I-e^{\bar{Q}_{\vec{\gamma}}t}){\rm diag}\{\vec{\alpha}\}\vec{\pi}.$

We now consider the steady-state regime (i.e., $t\to\infty$).
From the above expressions, it immediately follows that $\vec{\nu}_\infty=-\bar{Q}_{\vec{\gamma}}^{-1}{\rm diag}\{\vec{\alpha}\}\vec{\pi}$, and $\mu_\infty =\vec{1}^{\rm T}\vec{\nu}_\infty = -\vec{1}^{\rm T}\bar{Q}_{\vec{\gamma}}^{-1}{\rm diag}\{\vec{\alpha}\}\vec{\pi}$.
We further note that $\vec{\gamma}=-(Q-{\rm diag}\{\vec{\gamma}\})\vec{1}$, and hence $\vec{\gamma}^{\rm T}\bar{Q}_{\vec{\gamma}}^{-1}=-\vec{1}^{\rm T}$, so that $\vec{\gamma}^{\rm T}\vec{\nu}_\infty= \vec{\pi}^{\rm T}\vec{\alpha}.$

\item[$\rhd$] The variance can be found in a similar way. Define $\bar{\vec{Y}}(t):=\vec{Z}(t) M^2(t)$, and $\vec{w}_t:={\mathbb E}\bar{\vec{Y}}(t).$
Now our starting point is the relation
\[{\rm d}(M(t)-\mu_t)=\left(\vec{\alpha}^{\rm T}(\vec{Z}(t)-\vec{p}_t)-\vec{\gamma}^{\rm T}(\vec{Y}(t)-\vec{\nu}_t)\right){\rm d}t +\vec{\sigma}^{\rm T}\vec{Z}(t){\rm d}B(t),\]
so that
\begin{eqnarray*}
{\rm d}(M(t)-\mu_t)^2&=&2(M(t)-\mu_t)\left(\vec{\alpha}^{\rm T}(\vec{Z}(t)-\vec{p}_t)-\vec{\gamma}^{\rm T}(\vec{Y}(t)-\vec{\nu}_t)\right){\rm d}t \\
&&+\:2(M(t)-\mu_t) \vec{\sigma}^{\rm T}\vec{Z}(t)\,{\rm d}B(t)\,+\, \vec{\sigma}^{\rm T}{\rm diag}\{\vec{Z}(t)\}\vec{\sigma}\,{\rm d}t.
\end{eqnarray*}
Taking expectations of both sides,
\[v_t'=2\vec{\alpha}^{\rm T}\vec{\nu}_t-2\mu_t\vec{\alpha}^{\rm T}\vec{p}_t -2\vec{\gamma}^{\rm T}\vec{w}_t+2\mu_t\vec{\gamma}^{\rm T}\vec{\nu}_t+\vec{\sigma}^{\rm T}{\rm diag}\{\vec{p}_t\}\vec{\sigma}.\]
Clearly, to evaluate this expression, we first need to identify $\vec{w}_t$. To this end, we set up and equation for ${\rm d}\bar{\vec{Y}}(t)$ as before, take expectations, so as to obtain
\[\vec{w}'_t =\bar{Q}_{2\vec{\gamma}} \vec{w}_t + 2\,{\rm diag}\{\vec{\alpha}\}\vec{\nu}_t + {\rm diag}\{\vec{\sigma}^2\}\vec{p}_t;\]
here $\vec{\sigma}^2$ is the vector $(\sigma_1^2,\ldots,\sigma_d^2)^{\rm T}$. 
This leads to
\begin{equation}
\label{WW}\vec{w}_t=e^{\bar{Q}_{2\vec{\gamma}}t}
\vec{w}_0+
\int_0^t e^{\bar Q_{\vec{2\gamma}}(t-s)}\left(
2\,{\rm diag}\{\vec{\alpha}\}\vec{\nu}_s + {\rm diag}\{\vec{\sigma}^2\}\vec{p}_s\right)
{\rm d}s,\end{equation} so that $v_t=\vec{1}^{\rm T}{\vec{w}_t}-\mu_t^2.$ Observe that $\vec{w}_0=m_0^2\vec{p}_0.$

Again simplifications can be made if $\vec{p}_0=\vec{\pi}$ (and hence $\vec{p}_t=\vec{\pi}$ for all $t\geqslant 0$). In that case, we had already found an expression for $\vec{\nu}_s$ above, and as a result (\ref{WW}) can be explicitly evaluated.

For the stationary situation ($t\to\infty$, that is)
we obtain \
\[\vec{w}_\infty = -\bar{Q}_{2\vec{\gamma}}^{-1}\left(2\,{\rm diag}\{\vec{\alpha}\}\vec{\nu}_\infty +
{\rm diag}\{\vec{\sigma}^2\}\vec{\pi}\right),\]
and $v_\infty=\vec{1}^{\rm T}{\vec{w}_\infty}-\mu_\infty^2.$
\end{itemize}

We consider now an even more special case: $\gamma_i\equiv\gamma$ for all $i$ (in addition to $\vec{p}_t={\vec{\pi}}$; we let $t\geqslant 0$). It is directly seen that $\mu_\infty= \vec{\pi}^{\rm T}\vec{\alpha}/\gamma.$
Note that $\vec{\gamma}^{\rm T}\bar{Q}_{\vec{\gamma}}^{-1}=-\vec{1}^{\rm T}$ implies
$\vec{1}^{\rm T}\bar{Q}_{\delta\vec{1}}^{-1}=-\delta^{-1}\vec{1}^{\rm T}$ for any $\delta>0$, so that
\begin{eqnarray*}v_\infty&=&\vec{1}^{\rm T}{\vec{w}_\infty}-\mu_\infty^2=\frac{\vec{1}^{\rm T}{\rm diag}\{\vec{\alpha}\}\vec{\nu}_\infty}{\gamma}+\frac{\vec{\pi}^{\rm T}\vec{\sigma}^2}{2\gamma}-\left(\frac{\vec{\pi}^{\rm T}\vec{\alpha}}{\gamma}\right)^2\\
&=&-\frac{\vec{1}^{\rm T}{\rm diag}\{\vec{\alpha}\} \bar{Q}_{\gamma\vec{1}}^{-1}{\rm diag}\{\vec{\alpha}\}\vec{\pi}  }{\gamma}+\frac{\vec{\pi}^{\rm T}\vec{\sigma}^2}{2\gamma}-\left(\frac{\vec{\pi}^{\rm T}\vec{\alpha}}{\gamma}\right)^2
.\end{eqnarray*}
Now observe that, with $\check{D}_{ij}(\gamma):=\int_0^\infty p_{ij}(v)e^{-\gamma v}{\rm d}v$ for $\gamma>0$,
integration by parts yields
\[Q\check{D}(\gamma)=
\int_0^\infty QP(v)e^{-\gamma v}{\rm d}v= \int_0^\infty P'(v)e^{-\gamma v}{\rm d}v=
-I+\int_0^\infty \gamma P(v)e^{-\gamma v}{\rm d}v= -I+\gamma \check{D}(\gamma).
\]
As a consequence, $-(Q-\gamma I)\check{D}(\gamma) = I,$ so that
\[v_\infty= \frac{\vec{\pi}^{\rm T}\vec{\sigma}^2}{2\gamma}
+\frac{1}{\gamma}\vec{\alpha}^{\rm T}{\rm diag}\{\vec{\pi}\}\check{D}(\gamma)\vec{\alpha}
-\left(\frac{\vec{\pi}^{\rm T}\vec{\alpha}}{\gamma}\right)^2,\]
which, with $D_{ij}(\gamma):=\int_0^\infty (p_{ij}(v)-\pi_j)e^{-\gamma v}{\rm d}v=
\check{D}_{ij}(\gamma)-\pi_j/\gamma$, eventually leads to
\[v_\infty= \frac{\vec{\pi}^{\rm T}\vec{\sigma}^2}{2\gamma}
+\frac{1}{\gamma}\vec{\alpha}^{\rm T}{\rm diag}\{\vec{\pi}\}{D}(\gamma)\vec{\alpha}
.\]
The next subsection further studies the case in which the $\gamma_i$s are equal, i.e.,
$\gamma_i\equiv \gamma$, and the background process is in steady state at time $0$, i.e.,  and $\vec{p}_t=\vec{\pi}$.
As it turns out, under these conditions the mean and variance can also be found by an alternative elementary, insightful argumentation.

\subsection{Mean and variance: special case of equal ${\boldsymbol \gamma}$, starting in equilibrium}\label{SPECC}
In this subsection we consider the special case $\gamma_i\equiv \gamma$ for all $i$, while the Markov chain $X(t)$ starts off in equilibrium at time $0$ (so that ${\mathbb P}(X(t)=i)={\mathbb P}(X(0)=i)=\pi_i$ for all $t\geqslant 0$). In this special case we can evaluate $\mu_t$ and $v_t$ rather explicitly, particularly when in addition particular scalings are imposed.

We first concentrate on computing the transient mean $\mu_t$.
We denote by $X$ the path $(X(s), s\in [0,t]).$
Now using the representation of Thm.\ \ref{PROP}, and recalling the standard fact that $\mu_t$ can be written as ${\mathbb E} \,(\,{\mathbb E}(M(t)\,|\,X))$,
it is immediately seen  that $\mu_t$ can be written as a convex mixture of $m_0$ and ${\boldsymbol\pi}^{\rm T}{\boldsymbol\alpha}/\gamma$:
\[\mu_t = m_0e^{-\gamma t} + e^{-\gamma t}\int_0^t e^{\gamma s}{\rm d}s
\left(\sum_{i=1}^d \pi_i\alpha_i  \right)=m_0e^{-\gamma t} +\frac{ {\boldsymbol\pi}^{\rm T}{\boldsymbol\alpha}}{\gamma}\,(1-e^{-\gamma t});\]
use that $(X(t))_{t\geqslant 0}$ started off in equilibrium at time $0$. This expression
converges, as $t\to\infty$, to the stationary mean ${\boldsymbol\pi}^{\rm T}{\boldsymbol\alpha}/\gamma$, as expected.

The variance $v_t$ can be computed similarly, relying on the so called {\it law of total variance}, which says that ${\mathbb V}{\rm ar}\,M(t)=
{\mathbb E}\left({\mathbb V}{\rm ar}(M(t)\,|\,X)\right)+{\mathbb V}{\rm ar}\left({\mathbb E}(M(t)\,|\,X)\right)$.
Regarding the first term, it is seen that Thm.\ \ref{PROP} directly yields
\begin{eqnarray*}{\mathbb E}\left({\mathbb V}{\rm ar}(M(t)\,|\,X)\right)&=&{\mathbb E}\left(\int_0^t
e^{-2\gamma(t-s)}\sigma^2_{X(s)}\,{\rm d}s \right)\\
&=&\int_0^t
e^{-2\gamma(t-s)}{\mathbb E}\left(\sigma^2_{X(s)}\right)\,{\rm d}s=\sum_{i=1}^d \pi_i\sigma_i^2\left(\frac{1-e^{-2\gamma t}}{2\gamma}\right).
\end{eqnarray*}
Along similar lines,
\begin{eqnarray*}{\mathbb V}{\rm ar}\left({\mathbb E}(M(t)\,|\,X)\right)&=&
{\mathbb V}{\rm ar}\left(\int_0^t e^{-\gamma(t-s)}\alpha_{X(s)}\,{\rm d}s\right)\\&=&
\int_0^t\int_0^t {\mathbb C}{\rm ov}\left(e^{-\gamma(t-s)}\alpha_{X(s)},
 e^{-\gamma(t-u)}\alpha_{X(u)}\right){\rm d}u\,{\rm d}s\\&=&
e^{-2\gamma t}\int_0^t\int_0^t e^{\gamma (s+u)}{\mathbb C}{\rm ov}\left(\alpha_{X(s)},
\alpha_{X(u)}\right){\rm d}u\,{\rm d}s.
\end{eqnarray*}
The latter integral expression can be made more explicit. Recalling that $(X(t))_{t\geqslant 0}$ started off in equilibrium at time $0$, it can be evaluated as
\begin{eqnarray*}\lefteqn{2e^{-2\gamma t}\int_0^t\int_0^s e^{\gamma (s+u)}{\mathbb C}{\rm ov}\left(\alpha_{X(s)},
\alpha_{X(u)}\right){\rm d}u\,{\rm d}s}\\
&=& 2e^{-2\gamma t}\int_0^t\int_0^s e^{\gamma (s+u)}\sum_{i=1}^d\sum_{j=1}^d\alpha_i\alpha_j
\pi_i(p_{ij}(s-u)-\pi_j){\rm d}u\,{\rm d}s\\
&=&\frac{1}{\gamma} \sum_{i=1}^d\sum_{j=1}^d\alpha_i\alpha_j
\int_0^t \left(e^{-\gamma v}-e^{-\gamma(2t-v)}\right)\pi_i(p_{ij}(v)-\pi_j){\rm d}v
\end{eqnarray*}
(where the last equation follows after changing the order of integration and some elementary calculus). We arrive at the following result.
\begin{proposition} \label{PROP2}
For $t\ge 0$,
\[\mu_t = m_0e^{-\gamma t} +\frac{ {\boldsymbol\pi}^{\rm T}{\boldsymbol\alpha}}{\gamma}\,(1-e^{-\gamma t}),\]
and
\[v_t =\sum_{i=1}^d \pi_i\sigma_i^2\left(\frac{1-e^{-2\gamma t}}{2\gamma}\right)+
 \sum_{i=1}^d\sum_{j=1}^d\alpha_i\alpha_j
\int_0^t \left(\frac{e^{-\gamma v}-e^{-\gamma(2t-v)}}{\gamma}\right)\pi_i(p_{ij}(v)-\pi_j){\rm d}v.\]
\end{proposition}

We conclude this section by considering two specific limiting regimes, to which we return in Section~\ref{Sec:TS} where we will derive limit distributions under parameter scalings.
\begin{itemize}
\item[$\rhd$]
Specializing to the situation that $t\to\infty$, we obtain
\[{\mathbb V}{\rm ar}\, M =\frac{ {\boldsymbol\pi}^{\rm T}{\boldsymbol\sigma}^2}{2\gamma}+
\frac{1}{\gamma}
\sum_{i=1}^d\sum_{j=1}^d\alpha_i\alpha_j \pi_iD_{ij}(\gamma)=
\frac{ {\boldsymbol\pi}^{\rm T}{\boldsymbol\sigma}^2}{2\gamma}+
\frac{1}{\gamma}
\vec{\alpha}^{\rm T}{\rm diag}\{\vec{\pi}\}D(\gamma)\vec{\alpha},\]
in accordance with the expression we found before.
\item[$\rhd$] Scale ${\boldsymbol\alpha}\mapsto N^h{\boldsymbol\alpha}$, ${\boldsymbol\sigma^2}\mapsto N^h{\boldsymbol\sigma^2}$, and $Q\mapsto NQ$ for some $h\ge 0$. We obtain that ${\mathbb V}{\rm ar} \,M(t)$ equals
\[
N^h\sum_{i=1}^d \pi_i\sigma_i^2\left(\frac{1-e^{-2\gamma t}}{2\gamma}\right)\hspace{-0.4mm}+{N^{2h}}\sum_{i=1}^d\sum_{j=1}^d\alpha_i\alpha_j\hspace{-0.4mm}
\int_0^t \left(\frac{e^{-\gamma v}-e^{-\gamma(2t-v)}}{\gamma}\right)\hspace{-0.4mm}\pi_i\hspace{-0.2mm}\left(p_{ij}(vN)-\pi_j\right)\hspace{-0.2mm}
{\rm d}v,\]
which for $N$ large behaves as
\begin{eqnarray}
\lefteqn{\hspace{-15mm}\left(\frac{1-e^{-2\gamma t}}{2\gamma}\right)\left(N^h\sum_{i=1}^d \pi_i\sigma_i^2+2N^{2h-1}\sum_{i=1}^d\sum_{j=1}^d\alpha_i\alpha_j\pi_iD_{ij}\right)}\nonumber\\&=&
\left(\frac{1-e^{-2\gamma t}}{2\gamma}\right)\left(N^h\vec{\pi}^{\rm T}\vec{\sigma}^2+2N^{2h-1}
\vec{\alpha}^{\rm T}{\rm diag}\{\vec{\pi}\}D\vec{\alpha}\right)\label{PD}
,\end{eqnarray}
where $D:=D(0)$ is the deviation matrix introduced in Section \ref{MOD}.

We observe an interesting dichotomy: for $h<1$ the variance is essentially linear in the `scale' of the {\sc ou} processes $N^h$, while for $h>1$ it behaves superlinearly in $N^h$ (more specifically, proportionally to $N^{2h-1}$). It is this dichotomy that also featured in earlier work on Markov-modulated infinite-server queues \cite{BMT}.

The intuition behind the dichotomy is the following. If $h<1$, then the timescale of the background process systematically exceeds that of the $d$ underlying {\sc ou} processes (that is, the background process is `faster'). As a result, the system essentially behaves as an {\it ordinary} (that is, non-modulated) {\sc ou} process with `time average' parameters $\alpha_\infty:={\boldsymbol \pi}^{\rm T}{\boldsymbol\alpha}$, $\gamma$, and $\sigma_\infty^2:={\boldsymbol \pi}^{\rm T}{\boldsymbol\sigma^2}$.
If $h>1$, on the contrary, the background process jumps at a slow rate, relative to the typical timescale of the {\sc ou} processes; as a result, the process $(M(t))_{t\geqslant 0}$ moves between multiple local limits (where the individual `variance coefficients' $\sigma_i^2$ do not play a role). 
\end{itemize}

Note that it follows from (\ref{PD}) that ${\rm diag}\{\vec{\pi}\}D$ is a nonnegative definite matrix, although singular and non-symmetric in general; more precisely, it is a consequence of the fact that (\ref{PD}) is a variance and hence nonnegative, in conjunction with the fact the we can pick 
${\boldsymbol\sigma^2}={\boldsymbol 0}$. Below we state and prove the nonnegativity by independent arguments; cf.\ \cite[Prop 3.2]{DAVE}.

\begin{proposition}\label{prop:nnd} 
The matrix $D^\mathrm{T}\diag\{\vec{\pi}\}+\diag\{\vec{\pi}\}D$ is symmetric and nonnegative definite. 
\end{proposition}

\begin{proof}
First we prove the claim that the matrix $Q^\mathrm{T}\diag\{\vec{\pi}\}+\diag\{\vec{\pi}\}Q$ is (symmetric and) nonpositive definite. To that end we start from the semimartingale decomposition \eqref{eq:zz} for $\vec{Z}$.
By the product rule we obtain, collecting all the martingale terms in $\dd M(t)$,
\[
\dd (\vec{\vec{Z}}(t)\vec{Z}(t)^\mathrm{T})= Q^\mathrm{T} \vec{Z}(t)\vec{Z}(t)^\mathrm{T}\,\dd t + \vec{Z}(t)\vec{Z}(t)^\mathrm{T} Q\,\dd t +\dd\langle \vec{Z}\rangle_t +\dd M(t).
\]
As the predictable quadratic variation of $\vec{Z}$ is absolutely continuous and increasing, we can write $\dd\langle \vec{Z}\rangle_t=P_t\,\dd t$, where $P_t$ is a nonnegative definite matrix. Next we make the obvious observation that $\vec{Z}(t)\vec{Z}(t)^\mathrm{T}=\diag\{\vec{Z}(t)\}$. Hence we have by combining  \eqref{eq:zz} and the above display
\[
\diag\{Q^\mathrm{T} \vec{Z}(t)\}=Q^\mathrm{T}\diag\{\vec{Z}(t)\}+\diag\{\vec{Z}(t)\}Q+P_t.
\]
Taking expectations w.r.t.\ the stationary distribution of $\vec{Z}_t$ and using $Q^\mathrm{T}\pi=0$, we obtain
\[
0=Q^\mathrm{T}\diag\{\vec{\pi}\}+\diag\{\vec{\pi}\}Q+\ee P_t,
\]
from which it follows that $Q^\mathrm{T}\diag\{\vec{\pi}\}+\diag\{\vec{\pi}\}Q$ is (symmetric and) nonpositive definite. 

This in turn implies that $-D^\mathrm{T}(Q^\mathrm{T}\diag\{\vec{\pi}\}+\diag\{\vec{\pi}\}Q)D$ is symmetric and nonnegative definite. 
Recall now that $FQ=\Pi-I$ and hence $DQ=\Pi-I$.
Then $D^\mathrm{T} Q^\mathrm{T}\diag\{\vec{\pi}\}D= -(\diag\{\vec{\pi}\}-\vec{\pi}\vec{\pi}^\mathrm{T})D$. But $\vec{\pi}^\mathrm{T} D=0$, so $D^\mathrm{T} Q^\mathrm{T}\diag\{\vec{\pi}\}D= -\diag\{\vec{\pi}\}D$. The result now follows.
\end{proof}

\subsection{Covariances}
In this subsection we point out how to compute the covariance
\[c(t,u) :={\mathbb C}{\rm ov}\,(M(t), M(t+u)),\]
for $t,u\geqslant 0.$ To this end, we observe that by applying a time shift, we first assume in the computations to follow that $t=0$ ,and we consider $c(t):=\cov(M(t),M(0))$. Below we make frequently use of the additional quantities $\vec{C}(t)=\cov(\vec{Y}(t),M(0))$ and $\vec{B}(t)=\cov(\vec{Z}(t),M(0))$. Note that $c(t)=\one^{\rm T} \vec{C}(t)$. Multiplying Equations \eqref{eq:zz} and~\eqref{eq:y} by $M(0)$, we obtain upon  taking expectation the following system of {\sc ode}\,s:
\[
\begin{pmatrix}
\vec{B}'(t) \\
\vec{C}'(t)
\end{pmatrix}
=R\begin{pmatrix}
B(t) \\
C(t)
\end{pmatrix},\:\:\:\:\mbox{where}\:\:\:\:
R:=
\begin{pmatrix}
Q^{\rm T} & 0 \\
\diag\{\vec{\alpha}\} & \bar{Q}_{\vec{\gamma}}
\end{pmatrix}
\begin{pmatrix}
\vec{B}(t) \\
\vec{C}(t)
\end{pmatrix}
\]
with initial conditions $\vec{B}(0)=\cov(\vec{Z}(0),M(0))$ and $\vec{C}(0)=\cov(\vec{Y}(0),M(0))$. In a more compact and obvious notation, we have
$
\vec{A}'(t)=RA(t), 
$
and hence $
\vec{A}(t)
=\exp(Rt)
\vec{A}(0)$.

Likewise we can compute 
\[
A(t,u) :=
\begin{pmatrix}
\cov(\vec{Z}(t+u),M(t)) \\
\cov(\vec{Y}(t+u),M(t))
\end{pmatrix} =\exp(Ru)
\begin{pmatrix}
\cov(\vec{Z}(t),M(t)) \\
\cov(\vec{Y}(t),M(t))
\end{pmatrix} .
\]
It remains to derive an expression for the last covariances. For $\cov(\vec{Z}(t),M(t))$ we need $\ee M(t)\vec{Z}(t)=\ee \vec{Y}(t)$ and $\ee M(t)$, $\ee\vec{Z}(t)$. For $\cov(\vec{Y}(t),M(t))$ we need $\ee M(t)\vec{Y}(t)=\ee M(t)^2\vec{Z}(t)$, $\ee \vec{Y}(t)$ and $\ee M(t)$. All these quantities have  been obtained in Section~\ref{section:gc}.

\subsection{Recursive scheme for higher order moments}\label{ss:rec}

The objective of this section is to set up a recursive scheme to generate all transient moments, that is, the expected value of $M(t)^k$, for any $k\in\{1,2,\ldots\}$, jointly with the indicator function $1\{X(t)=i\}$. To that end we consider the expectation of $(M(t))^k\,\vec{Z}(t)$. 
First we rewrite Equation~\eqref{SDE:MMOU} as
\begin{equation}\label{eq:mmouz}
{\rm d} M \left( t \right) = \left( \vec{\alpha}^{\rm T} \vec{Z}(t) - \vec{\gamma}^{\rm T} \vec{Z}(t)X(t)\right) \, {\rm d} t + \vec{\sigma}^{\rm T} \vec{Z}(t) \, {\rm d} B \left( t \right).
\end{equation}
It\^o's lemma and \eqref{eq:mmouz}
directly yield
\begin{align*}
{\rm d}(M(t))^k & = k(M(t))^{k-1}\left( \vec{\alpha}^{\rm T} \vec{Z}(t) - \vec{\gamma}^{\rm T} \vec{Z}(t)X(t)\right) \, {\rm d} t + k(M(t))^{k-1}\vec{\sigma}^{\rm T} \vec{Z}(t) \, {\rm d} B \left( t \right) \\
& \qquad +\frac{1}{2}k(k-1)(M(t))^{k-2}\vec{\sigma}^{\rm T}\diag\{\vec{Z}(t)\}\vec{\sigma}\,{\rm d} t.
\end{align*}
Then we apply the product rule to $M(t)^k\vec{Z}(t)$, together with the just obtained equation and Equation~\eqref{eq:zz}, so as to obtain 
\begin{eqnarray*}
{\rm d}\left((M(t))^k\vec{Z}(t)\right) & =& k(M(t))^{k-1}\left( \diag\{\vec{\alpha}\}\vec{Z}(t) - \diag\{\vec{\gamma}\} \vec{Z}(t)M(t)\right) \, {\rm d} t \\
&&\hspace{-2mm}+ \:k(M(t))^{k-1}\diag\{\vec{\sigma}\}\vec{Z}(t) \, {\rm d} B \left( t \right) +\frac{1}{2}k(k-1)(M(t))^{k-2}\diag\{\vec{\sigma}^2\}\vec{Z}(t)\,{\rm d} t \\&&\hspace{-2mm}+\: (M(t))^k\left(Q^{\rm T} \vec{Z}(t)\,{\rm d}t+{\rm d}K(t)\right).
\end{eqnarray*}
All martingale terms on the right are genuine martingales and thus have expectation zero. Putting $\vec{H}_k(t):=\mathbb{E} M(t)^k\vec{Z}(t)$, we get the following recursion in {\sc ode} form:
\begin{eqnarray*}
\frac{{\rm d}}{{\rm d}t}\vec{H}_k(t) & =&k\diag\{\vec{\alpha}\}\vec{H}_{k-1}(t)-k\diag\{\vec{\gamma}\}\vec{H}_k(t)+
\frac{1}{2}k(k-1)\diag\{\vec{\sigma}^2\}\vec{H}_{k-2}(t) + Q^{\rm T} \vec{H}_k(t) \\
& =& \bar{Q}_{k\gamma}\vec{H}_k(t) + k\diag\{\vec{\alpha}\}\vec{H}_{k-1}(t)+
\frac{1}{2}k(k-1)\diag\{\vec{\sigma}^2\}\vec{H}_{k-2}(t).
\end{eqnarray*}
Stacking $\vec{H}_0(t),\ldots,\vec{H}_{n}(t)$ into a single vector $\bar{\vec{H}}_n(t)$, we obtain the differential equation
\[
\frac{{\rm d}}{{\rm d}t}\bar{\vec{H}}_n(t)=A_n\bar{\vec{H}}_n(t),
\]
with  $A_n\in\mathbb{R}^{(n+1)d\times (n+1)d}$ denoting a lower block triangular matrix, whose solution is $\vec{H}_n(t)=\exp(A_nt)\vec{H}_n(0)$.
Eventually, $h_k(t):=\mathbb{E}M(t)^k$ is given by $h_k(t)=\mathbf{1}^{\rm T} \vec{H}_k(t)$.
Note that for $k=1,2$ the results of Section~\ref{TB} can be recovered.


\section{Transient behavior: partial differential equations} \label{TBPDE}
The goal of this section is to characterize, for  a given vector  $\vec{t}\in{\mathbb R}^K$ (with $K\in{\mathbb N}$) such that $0\leqslant t_1\leqslant \cdots\leqslant t_K$, the Laplace transform of $(M(t+t_1), \ldots, M(t+t_K))$ (together with the state of the background process at these time instances). More specifically, we set up a system of {\sc pde}\,s for
\[g_{\vec{i}}(\vec{\vt},{t}) := {\mathbb E} e^{-(\vt_1M(t_1+t)+\cdots+ \vt_K M(t_K+t))} 1\{X(t_1+t)=i_1,\ldots, X(t_K+t)=i_K\};\]
here $t\geqslant 0$, $\vec{i}\in \{1,\ldots,d\}^K$ and $\vec{\vt}\in{\mathbb R}^K.$
The system of {\sc pde}\,s is with respect to $t$ and $\vt_1$ up to $\vt_K.$
We first point out the line of reasoning for the case $K=1$, and then present the {\sc pde}  for $K=2$. The cases $K\in\{3,4,\ldots\}$ can be dealt with fully analogously, but lead to notational inconveniences and are therefore left out.

It is noted that the stationary version of the result below (i.e., $t\to\infty$) for the special case $K=1$ has appeared in \cite{XING}
(where we remark that in \cite{XING} the additional issue of reflection at $0$ has been incorporated). 



\subsection{Fourier-Laplace transform}\label{FLT}

For $K=1$, the object of interest is
\[g_i(\vt,t) := {\mathbb E} e^{-\vt M(t)} 1\{X(t)=i\},\]
for $i=1,\ldots,d$; realize that, without loss of generality, we have taken $t_1=0.$ For a more compact notation we stack the $g_i$ in a single vector ${\boldsymbol g}$, so ${\boldsymbol g}(\vt,t) = {\mathbb E} e^{-\vt M(t)} \vec{Z}(t)$. Replacing in this expression $\vt$ by $-\ii u$ for $u\in\rr$ gives the characteristic function of $M(t)$ jointly with $\vec{Z}(t)$. 

\begin{theorem} \label{PDEth} Consider the case $K=1$ and $t_1=0$. The Laplace transforms ${\boldsymbol  g}(\vt,t)$ satisfy the following system of \,{\sc pde}\,s:

\begin{equation}\label{eq:cfpde}
\frac{\partial}{\partial t}{\boldsymbol g}(\vt,t)=Q^{\rm T}{\boldsymbol  g}(\vt,t)-\left(\vt\,{\rm diag}\{{\boldsymbol \alpha}\}-\frac{1}{2}\vt^2{\rm diag}\{{\boldsymbol \sigma}^2\}\right)
{\boldsymbol g}(\vt,t)
-\vt\,{\rm diag}\{{\boldsymbol \gamma}\}\,\frac{\partial}{\partial \vt}{\boldsymbol g}(\vt,t).
\end{equation}
The corresponding initial conditions are ${\boldsymbol g}(0,t)={\boldsymbol p}_t$ and ${\boldsymbol g}(\vt,0) = e^{-\vt m_0}{\boldsymbol p}_0.$
\end{theorem}

\begin{proof} The proof mimics the procedure used in Section \ref{ss:rec} to determine the moments of $M(t).$
Letting $f(\vt,t)=e^{-\vt M(t)}$, applying It\^o's formula to \eqref{eq:mmouz} yields
\[
\dd f(\vt,t)=-\vt f(\vt,t)\left(\left( \vec{\alpha}^{\rm T} \vec{Z}(t) - \vec{\gamma}^{\rm T} \vec{Z}(t)M(t)\right) \, {\rm d} t + \vec{\sigma}^{\rm T} \vec{Z}(t) \, {\rm d} B \left( t \right)\right)+\half\vt^2f(\vt,t)\diag\{\vec{\sigma}^2\}\vec{Z}(t)\dd t.
\]
We then apply the product rule to $f(\vt,t)\vec{Z}(t)$, using the just obtained equation in combination with Equation \eqref{eq:zz}. This leads to
\begin{align*}
\dd \left(f(\vt,t)\vec{Z}(t)\right)
& = -\vt f(\vt,t)\left(\left( \diag\{\vec{\alpha}\} \vec{Z}(t) - \diag\{\vec{\gamma}\} \vec{Z}(t)M(t)\right) \, {\rm d} t 
+ \diag\{\vec{\sigma}\} \vec{Z}(t) \, {\rm d} B \left( t \right)\right) \\\
& \qquad+\half\vt^2f(\vt,t)\diag\{\vec{\sigma}^2\}\vec{Z}(t)\dd t +
f(\vt,t)\left(Q^{\rm T} \vec{Z}(t)\,{\rm d}t+{\rm d}K(t)\right).
\end{align*}
Taking expectations, and recalling that  ${\boldsymbol g}(\vt,t)=\ee f(\vt,t)\vec{Z}(t)$ and that the martingale terms have expectation zero, we obtain
\begin{eqnarray*}
\frac{\partial}{\partial t} {\boldsymbol g}(\vt,t) & = &
-\vt \diag\{\vec{\alpha}\}{\boldsymbol g}(\vt,t)+\vt \diag\{\vec{\gamma}\} \ee\left(f(\vt,t)M(t)\vec{Z}(t)\right)   \\
&&+\:\half\vt^2\diag\{\vec{\sigma}^2\} f(\vt,t) +
Q^{\rm T} f(\vt,t).
\end{eqnarray*}
Realizing that ${\partial}{\boldsymbol g}/{\partial\vt}=-\ee\left(f(\vt,t)M(t)\vec{Z}(t)\right)$, we can rewrite this as \eqref{eq:cfpde}.
\end{proof}


It is remarked that the above system \eqref{eq:cfpde} of {\sc pde}\,s coincides, for $t\to\infty$, with the stationary result of \cite{XING} (where it is mentioned that in \cite{XING} the feature of reflection at 0 has been incorporated). In addition, it is noted that this system can be converted into a system of {\it ordinary} differential equations, as follows. Let $T$ be exponentially distributed with mean $\tau^{-1},$ independent of all other random features involved in the model.
Define
\[g_i(\vt) := {\mathbb E} e^{-\vt M(T)} 1\{X(T)=i\}.\]
Now multiply the {\sc pde} featuring in Thm.\ \ref{PDEth} by $\tau e^{-\tau t}$, and integrate over $t\in[0,\infty)$, to obtain (use integration by parts for the left-hand side)
\[\lambda\left({\boldsymbol g}(\vt)- e^{-\vt m_0}{\boldsymbol p}_0\right) = 
Q^{\rm T} {\boldsymbol g}(\vt) -
\left(\vt\,{\rm diag}\{{\boldsymbol \alpha}\}-\frac{1}{2}\vt^2{\rm diag}\{{\boldsymbol \sigma}^2\}\right)
{\boldsymbol g}(\vt)
-\vt\,{\rm diag}\{{\boldsymbol \gamma}\}\,\frac{\partial}{\partial \vt}{\boldsymbol g}(\vt).\]

\vb

All above results related to the case $K=1.$
For higher values of $K$ the same procedure can be followed; as announced we now present the result for $K=2$. Let $i,k$ be elements of $\{1,\ldots,d\}$, and $\vec{\vt}\equiv(\vt_1,\vt_2)\in{\mathbb R}^2$. We obtain the following system of {\sc pde}\,s:
\begin{eqnarray*}\lefteqn{\frac{\partial}{\partial t}g_{i,k}(\vec{\vt},t)=\sum_{j=1}^d  q_{ji}\, g_{j,k}(\vec{\vt},t)
+\sum_{\ell=1}^d  q_{\ell k}\, g_{i,\ell}(\vec{\vt},t)}\\
&&-\,\left(\vt_1\alpha_i+\vt_2\alpha_k-\frac{1}{2}\vt_1^2\sigma_i^2-\frac{1}{2}\vt_2^2\sigma_k^2\right)g_{i,k}(\vec{\vt},t)
-\vt_1\gamma_i\,\frac{\partial}{\partial \vt_1}g_{i,k}(\vec{\vt},t)
-\vt_2\gamma_k\,\frac{\partial}{\partial \vt_2}g_{i,k}(\vec{\vt},t),
\end{eqnarray*}
or in self-evident matrix notation, suppressing the arguments $\vec{\vt}$ and $t$,
\begin{eqnarray*}\lefteqn{\frac{\partial G}{\partial t}= Q^{\rm T}G+GQ-\vt_1 \,{\rm diag}\{{\boldsymbol \alpha}\} G
-\vt_2 G\,{\rm diag}\{{\boldsymbol \alpha}\}}\\
&&+\,\frac{1}{2}\vt_1^2 \,{\rm diag}\{{\boldsymbol \sigma}^2\} G
+\frac{1}{2}\vt_2^2 G\,{\rm diag}\{{\boldsymbol \sigma}^2\}
-\vt_1 \,{\rm diag}\{{\boldsymbol \gamma}\} \frac{\partial G}{\partial \vt_1}
-\vt_2  \frac{\partial G}{\partial \vt_2}\,{\rm diag}\{{\boldsymbol \gamma}\}.
\end{eqnarray*}

This matrix-valued system of {\sc pde}\,s can be converted into its vector-valued counterpart. Define the $d^2$-dimensional vector 
$\check{\boldsymbol g}(\vec{\vt},t):={\rm vec}(G(\vec{\vt},t)).$ Recall the definitions of the Kronecker sum
(denoted by  `$\oplus$') and the Kronecker product (denoted by `$\otimes$'). Using the relations ${\rm vec}(ABC)=(C^{\rm T}\otimes A){\rm vec}(B)$ and $A\oplus B= A\otimes I+I\otimes B$, for matrices $A$, $B$, and $C$ of appropriate dimensions,
we obtain the vector-valued {\sc pde}
\begin{eqnarray*}\lefteqn{\hspace{5mm}\frac{\partial \check {\vec{g}}}{\partial t}=
(Q^{\rm T}\oplus Q^{\rm T})\check{\vec{g}}
-\vt_1 (I\otimes {\rm diag}\{\vec{\alpha}\})\check{\vec{g}}
-\vt_2 ( {\rm diag}\{\vec{\alpha}\}\otimes I)\check{\vec{g}}}
 \\
&&+\,\frac{\vt_1^2}{2}  (I\otimes {\rm diag}\{\vec{\sigma}^2\})\check{\vec{g}}
+\frac{\vt_2^2 }{2}( {\rm diag}\{\vec{\sigma}^2\}\otimes I)\check{\vec{g}}
-\vt_1 (I\otimes {\rm diag}\{\vec{\gamma}\})\frac{\partial \check {\vec{g}}}{\partial \vt_1}
-\vt_2 ( {\rm diag}\{\vec{\gamma}\}\otimes I)\frac{\partial \check {\vec{g}}}{\partial \vt_2},
\end{eqnarray*}
again suppressing the arguments $\vec{\vt}$ and $t$. 

It is clear how this procedure should be extended to $K\in\{3,4,\ldots\}$,
but, as mentioned above, we do not include this because of the cumbersome notation needed.

\subsection{Explicit computations for two-dimensional case} We now present more explicit expressions relating to the case that $d=2$. Define $q:=q_1+q_2$. Suppose the system starts off at $(M(0),X(0)) = (m_0,2)$. Throughout this example we use the notation
\[g_i(\vt,t,j):= {\mathbb E}\left(e^{-\vt M(t)} 1\{X(t)=j\}\,|\,X(0)=i\right).\] 
The theory of this section yields the following system of partial differential equations:
\[\frac{\partial}{\partial t} g_{2}(t, \vt, 1)+\vt\gamma_1\,\frac{\partial}{\partial \vt} g_{2}(t, \vt, 1)=
\left(-q_{1}-\vt \alpha_{1}+\frac{1}{2}\vt^{2} \sigma_{1}^{2}\right) g_{2}(t, \vt, 1)+q_2 g_{2}(t, \vt, 2)
,\]
\[\frac{\partial}{\partial t} g_{2}(t, \vt, 2)+\vt\gamma_2\,\frac{\partial}{\partial \vt} g_{2}(t, \vt, 2)=
\left(-q_{2}-\vt \alpha_{2}+\frac{1}{2}\vt^{2} \sigma_{2}^{2}\right) g_{2}(t, \vt, 2)+q_1 g_{2}(t, \vt, 1)
,\]
with conditions (realizing that $\pi_i=q_i/q$)
\[
\left( \begin{array}{cc}
g_{2}(0, \vt, 1) \\  g_{2}(0, \vt, 2)\end{array} \right)
=\left( \begin{array}{cc}
0 \\  e^{-\vt x} \end{array} \right),
\:\:\:\:
\left( \begin{array}{cc}
g_{2}(t, 0, 1) \\  g_{2}(t, 0, 2)\end{array} \right)
=\left( \begin{array}{cc}
\pi_1-\pi_1e^{-qt}\\  \pi_2+\pi_1 e^{-qt}\end{array}\right),
\]
and $\vt \in {\mathbb R}$ and $t\in [0, \infty)$.

In the special case that $q_1=0$ (so that state 2 is transient, and state 1 is absorbing), the system of differential equations decouples; the second of the above two partial differential equations can be solved using the method of characteristics. Routine calculations lead to
\[
g_{2}(t, \vt, 2)=\exp \left(-\vt m_0 e^{-\gamma_{2} t}-q_{2}t-\frac{\alpha_{2}}{\gamma_{2}}(\vt-\vt e^{-\gamma_{2}t})+
\frac{\sigma^{2}_{2}}{4\gamma_{2}}(\vt^{2}-\vt^{2}e^{-2\gamma_{2}t})\right).
\]
Now the first equation of the two partial differential equations can be solved as well, with the distinguishing feature that now we have a non-homogeneous (rather than a homogeneous) single-dimensional partial differential equation. It can be verified that it is solved by 
\begin{eqnarray*}
g_{2}(t, \vt, 1)&=&
q_{2}\exp\left(-\frac{\alpha_{1}}{\gamma_{1}}\vt+ \frac{\sigma_{1}^{2}}{4\gamma_{1}}\vt^{2}\right)\times\\
&&\hspace{2cm}\int_{0}^{t}g_{2}(s, \vt e^{-\gamma_{1}(t-s)}, 2)
\exp\left(\frac{\alpha_{1}}{\gamma_{1}}\vt e^{-\gamma_{1}(t-s)}-\frac{\sigma_{1}^{2}}{4\gamma_{1}}\vt^{2}e^{-2\gamma_{1}(t-s)}\right)\DD s.
\end{eqnarray*}

\section{Parameter scaling}\label{Sec:TS}

So far we have characterized the distribution of $M(t)$ in terms of an algorithm to determine moments, and a {\sc pde} for the Fourier-Laplace transform. In other words, so far we have not presented any explicit results on the distribution of $M(t)$ itself. In this section we consider asymptotic regimes in which this {\rm is} possible; these regimes can be interpreted as parameter scalings.

More specifically, in this section we consider the following two scaled versions of the {\sc mmou} model.
\begin{itemize}
\item[$\rhd$] In the first we (linearly) speed up the background process (that is,
we replace  $Q\mapsto NQ$ or, equivalently, $X(t) \mapsto X(Nt)$). Our main result is that, as $N\to\infty$, the {\sc mmou} essentially experiences the time-averaged parameters, i.e., $\alpha_\infty:={\boldsymbol \pi}^{\rm T}{\boldsymbol\alpha}$, $\gamma_\infty:={\boldsymbol \pi}^{\rm T}{\boldsymbol\gamma}$ and $\sigma_\infty^2:={\boldsymbol \pi}^{\rm T}{\boldsymbol\sigma^2}$. As a consequence, it behaves as an {\sc ou} process with these parameters.
\item[$\rhd$] The second regime considered concerns a {\it simultaneous} scaling of the background process and the {\sc ou} processes. This is done as in Section \ref{SPECC}: 
$Q$ on the one hand, and  ${\boldsymbol\alpha}$ and ${\boldsymbol\sigma^2}$ on the other hand are scaled at {\it different} rates: we replace ${\boldsymbol\alpha}\mapsto N^h{\boldsymbol\alpha}$ and ${\boldsymbol\sigma^2}\mapsto N^h{\boldsymbol\sigma^2}$, but $Q\mapsto N Q$ for some $h\geqslant 0$). We obtain essentially two regimes, in line with the observations in Section \ref{SPECC}. 
\end{itemize}
As mentioned above, we are particularly interested in the limiting behavior in the regime that $N$ grows large. It is shown that the process $M(t)$, which we now denote as $M\hN(t)$ to stress the dependence on $N$, converges to the solution of a specific {\sc sde}. Importantly, we establish {\it weak convergence}, i.e., in the sense of convergence at the process level; our result can be seen as the counterpart of the result for Markov-modulated infinite-server queues in \cite{DAVE}. 

\vb

We consider sequences of {\sc mmou} processes, indexed by $\N$,
subject to the following scaling: \(Q \mapsto \N Q\); \(\balpha  \mapsto \N^h \balpha \); \(\bsigma  \mapsto \N^{h/2} \bsigma\),
where $h \geqslant 0$; note that by appropriately choosing $h$ we enter the two regimes described above as we let $N$ grow large  (see Corollaries \ref{C1} and \ref{C2}).
The definitions of $M(t)$, ${\m Z}(t)$ and ${\m K}(t)$ (the latter two having been defined in Section \ref{TB}) then take the following form
(where superscripts are being used to make the dependence on $\N$ and $h$ explicit):
\begin{equation} 
\dd M\tN(t) =  (\N^h \balpha -  \bgamma M\tN(t))^\rmt {\m Z}\sN(t)\, \dd t 
+ \N^{h/2}  \bsigma^\rmt {\m Z}\sN(t)\, \dd B(t),
\end{equation}
and
\begin{equation} \dd {\m Z}\sN(t) =  \N  Q^\rmt {\m Z}\sN(t)\,  \dd t + \dd {\m K}\sN(t)  . \label{eq:KsN} \end{equation}

We keep the initial condition $M\tN(0)$ at a fixed level $M(0)$. Let, with the definitions of  ${\alpha_\infty},
\gamma_\infty,$ and $\sigma^2_\infty$ given above, 
the `average path'  $\varrho(t)$ be defined by the {\sc ode}
\[ \dd \varrho(t) = (\alpha_\infty - \gamma_\infty \varrho(t)) \dd t,\,\, \varrho(0)=\one_{\{h=0\}}M(0), \]
such that we have
\[ \varrho(t) =  e^{-\gamma_\infty t}\varrho(0)+
\frac{\alpha_\infty}{\gamma_\infty} (1 - e^{-\gamma_\infty t}). \]
It is possible to show that $\varrho(t)$ coincides with $\lim_{N\to\infty}N^{-h}\ee M^{[N,h]}(t)$, in particular we have the initial value $\varrho(0)=\lim_{N\to\infty}\ee N^{-h}M(0)=\one_{\{h=0\}}M(0)$.

We can now state the main theorem of this section.

\begin{theorem}\label{THS} 
Under the scaling \(Q \mapsto \N Q\); \(\balpha  \mapsto \N^h \balpha \); \(\bsigma  \mapsto \N^{h/2} \bsigma \),
we have that the scaled and centered process \(\hat M\tN(t)\), as defined through
\[ \hat M\tN(t) := \N^{-\beta} (M\tN(t) - \N^h\varrho(t)), \]
converges weakly to the solution of the following {\sc sde}:
\[ \dd \hat M(t) = - \gamma_\infty \hat M(t) \dd t + \sqrt{ \sigma^2_\infty \one_{\{h\leqslant1 \}}+ {V}'(t) \one_{\{h\geqslant 1 \}}}\dd B(t),\,\:\:\:\:\hat M(0)=0.  \]
where \(\beta:=\max\{h/2,h-1/2\}\), $B$ a Brownian motion, and
\begin{equation}
\label{defv}
{V}(t) := \int_0^t (\balpha - \bgamma\varrho(s))^\rmt (\diag\{\bpi\} D+D^\rmt \diag\{\bpi\})(\balpha - \bgamma \varrho(s)){\rm d}s.
\end{equation}
\end{theorem}

Before proving this result, we observe that the above theorem provides us with the limiting behavior in the two regimes described at the beginning of this section. In the first corollary we simply take $h=0.$

\begin{corollary}\label{C1}
Under the scaling \(Q \mapsto \N Q\), with \(\balpha\) and \(\bsigma\) kept at their original values, we have that $M^{[N,0]}(t)$ converges 
weakly to a process ${\mathscr M}_1(t)$, which is an (ordinary, i.e., non-modulated)  {\sc ou} process with parameters \((\alpha_\infty, \gamma_\infty, \sigma_\infty)\), defined through the {\sc sde}
\[ \dd {\mathscr M}_1(t) = (\alpha_\infty - \gamma_\infty {\mathscr M}_1 (t))\dd t + \sigma_\infty \dd B(t). \]
\end{corollary}

The second corollary describes the situation in which both the background process and the {\sc ou} process are scaled, but at  different rates. We explicitly characterize the limiting behaviour in each of the three resulting regimes.

\begin{corollary}\label{C2}
Under the scaling \(Q \mapsto \N Q\); \(\balpha  \mapsto \N^h \balpha \); \(\bsigma  \mapsto \N^{h/2} \bsigma \), we have that $\hat M^{[N,h]}(t)$ converges 
weakly to a process ${\mathscr M}_2(t)$, defined through one of the following {\sc sde}\,s: if \,$0<h<1$, then
\[\dd  {\mathscr M}_2(t) = - \gamma_\infty  {\mathscr M}_2(t)\, \dd t +  \sigma_\infty \dd B(t),\] if $h=1$, then
\[\dd  {\mathscr M}_2(t) = - \gamma_\infty  {\mathscr M}_2(t)\, \dd t + \sqrt{\sigma^2_\infty +{V}'(t)}\dd B(t), \]
and if $h>1$, then
\[\dd {\mathscr M}_2(t) = - \gamma_\infty  {\mathscr M}_2 (t)\,\dd t +\sqrt{{V}'(t)} \, \dd B(t).\]
\end{corollary}

These corollaries are trivial consequences of Thm.\ \ref{THS}, and therefore  we direct our attention to the proof of this main theorem itself.
We remark that Corollary \ref{C2} confirms an observation we made in Section~\ref{TB}: for $h<1$ the system essentially behaves as an non-modulated {\sc ou} process, while for $h>1$ 
the background process plays a role through its deviation matrix $D$. 
\medskip\\
In the proof of Theorem~\ref{THS} we need an auxiliary result, which we present first. 
%

\begin{lemma}\label{lemmah}
Let the $d$-dimensional row vectors ${\boldsymbol \Psi}^{[N]}$ be a sequence of predictable processes such that ${\boldsymbol \Psi}^{[N]}(t)\to{\boldsymbol \Psi}(t)$ in probability uniformly on compact sets, i.e., as $N\to\infty$,
\[\sup_{t\leq T}|{\boldsymbol \Psi}^{[N]}(t)-{\boldsymbol \Psi}(t)|\to 0\] in probability for every $T>0$;  here ${\boldsymbol \Psi}$ is deterministic, satisfying $\int_0^t {\boldsymbol \Psi}(s){\boldsymbol \Psi}(s)^\rmt\,\dd s<\infty$ for every $t>0$. 
Furthermore, let $\vec{X}^{[N]}$ be continuous semimartingales that converge weakly to a $d$-dimensional scaled Brownian motion $\boldsymbol B$ with quadratic variation $\langle{\boldsymbol B}\rangle_t=Ct$ (where $C\in\rr^{d\times d}$).
%
Then, as $N\to\infty$,  the stochastic integrals \[\int_0^\cdot {\boldsymbol \Psi}^{[N]}(s)\,\dd \vec{X}^{[N]}(s)\]  converge  weakly to the time-inhomogeneous Brownian motion $B^{\boldsymbol \Psi}:=\int_0^\cdot {\boldsymbol \Psi}(s)\,\dd \boldsymbol B(s)$ with quadratic variation
\[
\langle B^{\boldsymbol \Psi}\rangle_t=\int_0^t{\boldsymbol \Psi}(s)C{\boldsymbol \Psi}(s)^\rmt\,\dd s.
\]

\end{lemma}

The claim of Lemma \ref{lemmah} essentially
follows from \cite[Thm. VI.6.22]{JS}. To check the condition of the cited theorem, one needs weak convergence of the pair $({\boldsymbol \Psi}^{[N]}, \vec{X}^{[N]})$, but this is guaranteed by the uniform convergence in probability of the ${\boldsymbol \Psi}^{[N]}(t)$.
\vb

We now proceed with the proof of Thm.\ \ref{THS}.

\begin{proof} 

The proof of Thm.\ \ref{THS} consists of 4 steps.

$\rhd$ {\it Step} 1. 
We describe the dynamics of the process $\hat M\tN(t)$ through
\begin{align*}
\dd\hat M\tN(t) & =N^{h-\beta}(\balpha-\rho(t)\bgamma)^\rmt(\vec{Z}^{[N]}(t)-\bpi)\,\dd t+N^{h/2-\beta}\vec{\bsigma}^\rmt\bz^{[N]}_t\dd B(t) - \bgamma^\rmt \bz^{[N]}_t\hat M\tN(t)\,\dd t \\
& =: N^{h-\beta-\half}\dd G^{[N]}(t)+N^{h/2-\beta}\dd\hat{B}^{[N]}(t)- \bgamma^\rmt \bz^{[N]}_t\hat M\tN(t)\,\dd t.
\end{align*}
Defining  $\bzetan(t):=\int_0^t\bz^{[N]}(s)\,\dd s$ and $\ynh(t):=e^{\bgamma^\rmt\bzetan(t)}\hat M\tN(t)$, one obtains
\begin{align}\label{eq:y}
\dd \ynh(t) & = N^{h-\beta-\half}e^{\bgamma^\rmt\bzetan_t}\dd G^{[N]}(t)+N^{h/2-\beta}e^{\bgamma^\rmt\bzetan_t}\dd\hat{B}^{[N]}(t). 
\end{align}
In the next two steps we analyze the two terms in the right hand side of \eqref{eq:y}.

\vb

$\rhd$ {\it Step} 2. 
We first consider the first term on the right hand side of \eqref{eq:y}. To analyze it, we need 
the functional central limit theorem for the martingale ${\vec{K}}_{\circ}\hN:=\vec{K}^{[N]}/\sqrt{N}.$ From the proof of Prop.~\ref{prop:nnd} we know that 
\[
\frac{1}{N}\langle {\vec{K}}_\circ^{[N]}\rangle_t=\int_0^t\left(\diag\{Q^\mathrm{T} \vec{Z}^{[N]}(s)\}-Q^\mathrm{T}\diag\{\vec{Z}^{[N]}(s)\}-\diag\{\vec{Z}^{[N]}(s)\}Q\right)\,\dd s,
\]
which by the ergodic theorem converges to $-Q^\mathrm{T}\diag\{\vec{\bpi}\}-\diag\{\vec{\bpi}\}Q$. As the jumps of ${\vec{K}}_\circ^{[N]}$ are of order $O(1/\sqrt{N})$, the martingale central limit theorem (see e.g.\ \cite[Thm.~VIII.3.11]{JS} or \cite[Thm.~7.1.4]{ETHIERKURTZ}) gives the weak convergence of ${\vec{K}}_\circ^{[N]}$ to a $d$-dimensional scaled Brownian motion ${\boldsymbol B}_\circ$ with \[\langle\boldsymbol{B}_\circ\rangle_t=-\left(Q^\mathrm{T}\diag\{\vec{\bpi}\}+\diag\{\vec{\bpi}\}Q\right)\,t.\] Moreover, we then also deduce the weak convergence of the process $\vec{Z}^{[N,Q]}:=\sqrt{N}\int_0^\cdot Q^{\rm T}\vec{Z}^{[N]}(s)\,{\rm d}s.$ to $-\boldsymbol{B}_\circ$, and hence to $\boldsymbol{B}_\circ$ as well. 

\begin{itemize}
\item[$\circ$]
We first apply Lemma \ref{lemmah} with the choice (with $D$ denoting the deviation matrix)
\[
{\boldsymbol \Psi}^{[N]}(t)  := -(\balpha-\rho(t)\bgamma)^\rmt D^\rmt, \:\:\:\:\:
X^{[N]}   := \vec{Z}^{[N,Q]},
\]
to
the process \[G^{[N]}=\sqrt{N}\int_0^\cdot (\balpha-\rho(s)\bgamma)^\rmt (\vec{Z}^{[N]}(s)-\bpi)\,{\rm d}s=-\sqrt{N}\int_0^\cdot (\balpha-\rho(s)\bgamma)^\rmt (QD)^\rmt \vec{Z}^{[N]}(s)\,{\rm d}s,\]
where the last equality follows from $QD=\one\bpi^\rmt-I$ (see the proof of Prop.~\ref{prop:nnd}). Note that ${\boldsymbol \Psi}^{[N]}(t)={\boldsymbol \Psi}(t)$ for all $N$, and therefore it is immediate that the weak limit can be identified as a continuous Gaussian martingale $G$, where it turns out that 
$\langle G\rangle_t=V(t),$ with $V(t)$ defined in (\ref{defv}), which again follows from the proof of Prop.~\ref{prop:nnd}.

\item[$\circ$]
In the next step we consider the processes $\int_0^\cdot{\Psi}^{[N]}(s)\,\dd G^{[N]}(s)$, with ${\Psi}^{[N]}(s):=\exp(\bgamma^\rmt \bzetan(s))$.  As these processes are increasing, we have the a.s.\ convergence of \[\sup_{s\leq T}\,\left|\,\exp(\bgamma^\rmt \bzetan(s))-\exp(\bgamma^\rmt \bpi\,s)\,\right|\to 0\] as $N\to\infty$, by combining the ergodic theorem with \cite[Thm.\ VI.2.15(c)]{JS} (which states that pointwise convergence of increasing functions to a continuous limit implies uniform convergence on compacts). As an immediate consequence of the above and Lemma~\ref{lemmah}, we obtain the weak convergence of $\int_0^\cdot \exp(\bgamma^\rmt \bzetan(s))\,\dd G^{[N]}(s)$ to $\int_0^\cdot \exp(\bgamma^\rmt \bpi\,s)\,\dd G(s)=\int_0^\cdot \exp(\gamma_\infty s)\,\dd G(s)$. 
\end{itemize}

\vb

$\rhd$ {\it Step} 3.
We now consider the second term on the right hand side of \eqref{eq:y}. For the Brownian term $\hat{B}^{[N]}$  we have by the martingale central limit theorem  weak convergence to the Gaussian martingale $\hat{B}$, with quadratic variation $\langle\hat{B}\rangle_t=\sigma^2_\infty t$. The convergence of $\int_0^\cdot \exp(\bgamma^{\rm T} \bzetan(s))\,\dd \hat{B}^{[N]}(s)$ can be handled as above to obtain weak convergence to the Gaussian martingale $\int_0^\cdot \exp(\bgamma^\rmt \bpi s)\,\dd \hat{B}(s)=\int_0^\cdot \exp(\gamma_\infty s)\,\dd \hat{B}(s)$.

\vb

$\rhd$ {\it Step} 4. 
In order to finally obtain the weak limit of $Y^{[N,h]}$ we use\[
h-\beta-\half =\frac{1}{2}\min\{h-1,0\},\:\:\:\:\:
\frac{h}{2}-\beta  =\frac{1}{2}\min\{1-h,0\}.
\]

Clearly, for $h<1$ we have convergence of $Y^{[N,h]}$ to $\int_0^\cdot \exp(\bgamma^\rmt \bpi s)\,\dd \hat{B}(s)$, whereas for $h>1$ we have convergence to  $\int_0^\cdot \exp(\bgamma^\rmt \bpi\,s)\,\dd G(s)$. For $h=1$ we get weak convergence to the sum of these. To see this, recall that the weak convergence of the $G^{[N]}$ was based on properties of the Markov chain, whereas the convergence of the $B^{[N]}$ resulted from considerations involving the Brownian motion $B$, and these basic processes are independent. Note further that $Y^{[N,h]}(0)=N^{-\beta}M(0)-N^{h-\beta}\one_{\{h=0\}}M(0)\to 0$. Combining these results, we find that $Y^{[N,h]}$ converges to a Gaussian martingale $Y$ given by
\[
 Y(t)= 
 \int_0^t e^{\gamma_\infty s}(\one_{\{h\leqslant 1\}}\,\dd \hat{B}(s)+\one_{\{h\geqslant 1\}}\,\dd G(s)),
\]
and hence the $\hat{M}^{[N,h]}$ converge weakly to the limit $\hat{M}$ given by $\hat{M}(t)=e^{-\gamma_\infty t}Y(t)$, and this process satisfies the {\sc sde}
\[
\dd \hat{M}(t)=-\gamma_\infty \hat{M}(t)\,\dd t+(\one_{\{h\leqslant 1\}}\,\dd \hat{B}(t)+\one_{\{h\geqslant 1\}}\,\dd G(t)).
\]
In this equation the (continuous, Gaussian) martingale has quadratic variation 
$\one_{\{h\leqslant 1\}}\sigma_\infty^2 t+
\one_{\{h\geqslant 1\}}V(t)
$. 
Hence we can identify its distribution with that of 
\[
\int_0^\cdot\sqrt{\one_{\{h\leqslant 1\}}\sigma_\infty^2+\one_{\{h\geqslant 1\}}{V}'(s)}\,\dd B(s), 
\]
where $B$ is a standard Brownian motion. This finishes the proof.
\end{proof}

\newcommand{\TT}{^{\rm T}}

\newcommand{\hj}{^{(j)}}
\newcommand{\hk}{^{(k)}}
\newcommand{\hjk}{^{(j,k)}}
\newcommand{{\ha}}{^{(1)}}
\newcommand{{\hb}}{^{(2)}}

\section{Multiple MMOU processes driven by the same background process}
{In this section, we consider a single background process $X$, taking as before values in $\{1,\ldots,d\}$, modulating {\it multiple} {\sc ou} processes. Suppose there are $J\in{\mathbb N}$ such processes, with parameters $(\vec{\alpha}^{(1)},\vec{\gamma}^{(1)},\vec{\sigma}^{(1)})$ up to
$(\vec{\alpha}^{(J)},\vec{\gamma}^{(J)},\vec{\sigma}^{(J)})$. It is further assumed that the {\sc ou} processes are driven by {\it independent} Brownian motions $B_1(\cdot).\ldots, B_J(\cdot)$.} Combining the above, this leads to the $J$ coupled {\sc sde}\,s
\[{\rm d}{M_j(t)} = \left(\alpha\hj_{X(t)}-\gamma\hj_{X(t)}{M_j(t)}\right){\rm d}t +\sigma\hj_{X(t)} \,{\rm d}B_j(t),\]
for $j=1,\ldots,J.$ We call the process a $J$-{\sc mmou} process.

Interestingly, this construction yields $J$ components that have common features, as they react to the same background process, as well as component-specific features, as a consequence of the fact that the driving Brownian motions are independent. This model is particularly useful in settings with multidimensional stochastic processes whose components are affected by the same external factors. 

An example of a situation where this idea can be exploited is that of
multiple asset prices reacting to the (same) state of the economy, which could be represented by a 
background process (for instance with two states, that is, alternating between a `good' and a `bad' state).
In this way the dependence between  the individual components can be naturally modeled. 
In mathematical finance, one of the key challenges is to develop models
that incorporate the correlation between the individual components in a sound way. Some proposals were to simplistic, ignoring
too many relevant details, while others correspond with models with overly many
parameters, with its repercussions in terms of the calibration that needs to be performed.

Another setting in which such a coupling may offer a natural modeling framework is that of a wireless network. Channel conditions may be modeled as alternating between various levels, and users' transmission rates may react in a similar way to these fluctuations.

\vb

Many of the results derived in the previous sections, covering the case $J=1$, can be generalized to the situation of    $J$-{\sc mmou} processes described above. To avoid unnecessary repetition, we
restrict ourselves to a few of these extensions. In particular, we present (i)~the counterpart of Thm.\ \ref{PROP}, stating that $\vec{M}(t)$ is, conditionally on the path of the background process,
multivariate Normally distributed; (ii)~some explicit calculations for the means and (co-)variances for certain special cases; (iii)~the generalization of the {\sc pde} of Thm.\ \ref{PDEth}, (iv) explicit expressions for the steady-state (mixed) moments. 
Procedures for transient moments, and scaling results (such as a $J$-dimensional {\sc clt}) are not included in this paper,
but can be developed as in the single-dimensional case.

\subsection{Conditional Normality}

First we condition on the path $(X(s), s\in [0,t]).$ It is evident that, under this conditioning, the individual components
of ${\vec{M}}(t)$ are independent. The following result describes this setting in greater detail.

\begin{proposition}\label{PROP2}
Define $\Gamma \hj(t) := \int_{0}^{t} \gamma\hj_{X (s)} \, {\rm d} s$, for $j=1,\ldots,J.$ Then the $J$-dimensional stochastic process $(\vec{M}(t))_{t \geqslant 0}$ given by
\begin{align*}
M\hj \left( t \right) = M\hj_0 e^{ - \Gamma\hj \left( t \right) } + \int_{0}^{t} e^{ - \left( \Gamma\hj \left( t \right) - \Gamma\hj \left( s \right) \right) } \alpha\hj_{X \left( s \right)} \, {\rm d} s 
+ \int_{0}^{t} e^{ - \left( \Gamma\hj \left( t \right) - \Gamma\hj \left( s \right) \right) } \sigma\hj_{X \left( s \right)} \, {\rm d} B \left( s \right)
\end{align*}
is the unique $J$-{\sc mmou} process with initial condition $\vec{M}_0$. 

Conditional on the process $X$, the random vector $\vec{M} \left( t \right)$ has a multivariate Normal distribution with, for $j=1,\ldots,J$, random mean
\[
\mu\hj \left( t \right) = M_0\hj \exp\left({-\Gamma \hj(t) }\right)+\int_0^t\exp\left({-
(\Gamma \hj(t) -\Gamma \hj(s)) }\right)\alpha\hj_{X(s)}\,{\rm d}s\]
and random covariance $v^{(j,k)}(t)=0$ if $j\not=k$ and
\[
v ^{(j,j)}\left( t \right) = \int_{0}^{t} 
\exp\left({-2
(\Gamma \hj(t) -\Gamma \hj(s)) }\right)\left(\sigma\hj_{X(s)}\right)^2\,{\rm d}s.
\]
\end{proposition}

\subsection{Mean and (co-)variance}

The mean and (co-)variance of $\vec{M}(t)$ for $J$-{\sc mmou} can be computed relying on stochastic integration theory, with a procedure similar to the one relied on in Section~\ref{TB}; we do not include the resulting expressions.

{We consider in greater detail the special case that $\gamma_i\hj\equiv \gamma\hj$ for all $i\in\{1,\ldots,d\}$ (as in Section~\ref{TB}), because in this situation expressions simplify greatly. The means and variances can be found as in Prop.~\ref{PROP}; we now point out how to compute the covariance $v_t\hjk:={\mathbb C}{\rm ov}(M\hj(t), M^{(k)}(t))$ (with $j\not=k$), relying on the {\it law of total covariance}. We write, in self-evident notation,
\[v_t\hjk={\mathbb E}({\mathbb C}{\rm ov}(M\hj(t), M^{(k)}(t)\,|\,X))+
{\mathbb C}{\rm ov}({\mathbb E}(M\hj(t)\,|\,X), {\mathbb E}(M^{(k)}(t)\,|\,X)).\]
The first term obviously cancels (cf.\ Prop.\ \ref{PROP2}), while the second reads
\begin{eqnarray*}\lefteqn{\hspace{-1cm}\frac{1}{\gamma\hj+\gamma\hk}\left(\sum_{i_1=1}^d\sum_{i_2=1}^d \alpha\hj_{i_1}\alpha_{i_2}\hk\int_0^t\left(e^{-\gamma\hk v}-e^{-(\gamma\hj+\gamma\hk)t+\gamma\hj v}\right)\pi_{i_1}(p_{i_1i_2}(v)-\pi_{i_2}){\rm d}v\right.}\\
&&\left.\sum_{i_1=1}^d\sum_{i_2=1}^d \alpha\hk_{i_1}\alpha_{i_2}\hj\int_0^t\left(e^{-\gamma\hj v}-e^{-(\gamma\hk+\gamma\hj)t+\gamma\hk v}\right)\pi_{i_1}(p_{i_1i_2}(v)-\pi_{i_2}){\rm d}v\right)
.\end{eqnarray*}
We consider two limiting regimes.
\begin{itemize}
\item[$\rhd$]
For $t\to\infty$, it is readily checked that there is convergence to
\[\frac{1}{\gamma\hj+\gamma\hk}\left((\vec{\alpha}\hj)^{\rm T}{\rm diag}\{\vec{\pi}\}D(\gamma\hk)\vec{\alpha}\hk+
(\vec{\alpha}\hk)^{\rm T}{\rm diag}\{\vec{\pi}\}D(\gamma\hj)\vec{\alpha}\hj\right).\]
\item[$\rhd$] Apply, as before, the scaling ${\boldsymbol\alpha}\mapsto N^h{\boldsymbol\alpha}$, ${\boldsymbol\sigma^2}\mapsto N^h{\boldsymbol\sigma^2}$, and $Q\mapsto NQ$ for some $h>0$. We obtain that the covariance,
for $N$ large, behaves as
\[\left(\frac{1-e^{-(\gamma \hj+\gamma\hk)t}}{\gamma\hj+\gamma\hk}\right)\left(2N^{2h-1}(\vec{\alpha}\hj)^{\rm T}{\rm diag}\{\vec{\pi}\}D\vec{\alpha}\hk\right).\]
\end{itemize}}

{\begin{Ex} {\rm 
We now provide explicit results for $t\to\infty$ for the case $d=2,J=2$.
It can be verified that, with $q_1:=q_{12}$, $q_2:=q_{21}$ and $q:=q_1+q_2$,
\[D(\gamma\hj)=\frac{1}{q(q+\gamma\hj)}\left(\begin{array}{rr}
q_2&-q_2\\
-q_1&q_1\end{array}\right).\]
It is a matter of elementary calculus to show that the steady-state covariance is
\[{\mathbb C}{\rm ov}(M_1, M_2)
=\frac{1}{\gamma\ha+\gamma\hb}\frac{q_1q_2}{q^2}\frac{2q+\gamma\ha+\gamma\hb}{(q+\gamma\ha)(q+\gamma\hb)}
\left((\alpha_1\hb-\alpha_2\hb)(\alpha_1\ha-\alpha_2\ha)\right)
\]
whereas, for $j=1,2$,
\[{\mathbb V}{\rm ar}\,M_j=\frac{q_1(\sigma_2\hj)^2+q_2(\sigma_1\hj)^2}{2\gamma\hj q}+\frac{1}{\gamma\hj}\frac{q_1q_2}{q^2(q+\gamma\hj)}\left(\alpha_1\hj-\alpha_2\hj\right)^2.\]
These expressions enable us to compute the correlation coefficient between
$M_1$ and $M_2$. For the special case that $\vec{\sigma}\ha=\vec{\sigma}\hb=\vec{0}$, we obtain, modulo its sign,
\[\sqrt{\frac{\gamma\ha\gamma\hb}{(q+\gamma\ha)(q+\gamma\hb)}}\frac{2q+\gamma\ha+\gamma\hb}{\gamma\ha+\gamma\hb},\]
which can be verified to be smaller than 1.
}\end{Ex}}

\subsection{Transient behavior: partial differential equations}

In order to uniquely characterize the joint distribution of $\vec{M}(t)$, we now set up a system of partial differential equations for
the objects
\[g_i(\vec{\vt},t):= {\mathbb E} \left(e^{\sum_{j=1}^J \vt_j M\hj(t) } 1\{X(t)=i\}\right),\]
with $i\in\{1,\ldots,d\}.$
Relying on the machinery  used when establishing the system of {\sc pde}\,s featuring in Thm.\ \ref{PDEth}, we obtain that ${\partial {\boldsymbol g}(\vec{\vt},t)}/{\partial t}$ equals
\[Q^{\rm T}{\boldsymbol  g}(\vec{\vt},t)-
\sum_{j=1}^J\left(\vt_j\,{\rm diag}\{{\boldsymbol \alpha}\hj\}-\frac{1}{2}\vt_j^2{\rm diag}\{({\boldsymbol \sigma}\hj)^2\}\right)
{\boldsymbol g}(\vec{\vt},t)-\sum_{j=1}^J\vt_j\,{\rm diag}\{{\boldsymbol \gamma}\hj\}\,\frac{\partial}{\partial \vt_j}{\boldsymbol g}(\vec{\vt},t).
\]

\subsection{Recursive scheme for higher order moments}

The above system of {\sc pde}\,s can be used to determine all (transient and stationary) moments related to $J$-{\sc mmou}. We restrict ourselves to the stationary moments here. 
Define  $\vec{h}_{\vec{k}}= (h_{1,\vec{k}},\ldots,h_{d,\vec{k}})^{\rm T},$
where
\[h_{i,\vec{k}}:=
{\mathbb E}\left((-1)^{\sum_{j=1}^J k_j} (M^{(1)})^{k_1}\cdots (M^{(J)})^{k_J}1\{X=i\}\right).\]
Observe that $\vec{h}_{\vec{0}} =\vec{\pi}.$
With techniques similar to those applied earlier, $\vec{e}_j\in{\mathbb R}^J$ denoting the $j$-th unit vector, we obtain the recursion
\begin{eqnarray*}\vec{h}_{\vec{k}} &=& \left(Q\TT-\sum_{j=1}^J k_j\,{\rm diag}\{{\boldsymbol \gamma}\hj\}\right)^{-1}\\
&&\times\:
\left(\sum_{j=1}^J k_j\,{\rm diag}\{{\boldsymbol \alpha}\hj\} \vec{h}_{\vec{k}-\vec{e}_j}
-\frac{1}{2}\sum_{j=1}^J k_j(k_j-1)\,{\rm diag}\{({\boldsymbol \sigma}\hj)^2\} \vec{h}_{\vec{k}-2\vec{e}_j}\right).\end{eqnarray*}
This procedure allows us to compute all mixed moments, thus facilitating the calculation of covariances as well. 
In the situation of $J=2$, for instance, we find that
\[{\mathbb E}M^{(1)}M^{(2)} = \vec{1}\left(Q\TT-{\rm diag}\{{\boldsymbol \gamma}\ha\}-
{\rm diag}\{{\boldsymbol \gamma}\hb\}\right)^{-1} 
\left(
{\rm diag}\{{\boldsymbol \alpha}\ha\}\vec{h}_{0,1}+
{\rm diag}\{{\boldsymbol \alpha}\hb\}\vec{h}_{1,0}
\right),\]
where $\vec{h}_{0,1}$ and $\vec{h}_{1,0}$ follow from the analysis presented in Section 5.

\begin{remark} {\rm
The model proposed in this section describes a $J$-dimensional stochastic process with dependent components. In many situations, the dimension $d$ can be chosen relatively 
small  (see for instance \cite{BAN,DAV}), whereas $J$ tends to  be large (e.g., in the context of asset prices).
Importantly, the $\frac{1}{2}J(J+1)=O(J^2)$ entries of the covariance matrix of ${\vec{M}}(t)$ (or its stationary counterpart $\vec{M}$) are endogenously determined by the model, and need not be estimated from data.
Instead, this approach requires the calibration of just the $d(d-1)$ entries of the $Q$-matrix, as well as the $3dJ$ parameters of the underlying {\sc ou}
processes, totaling $O(J)$ parameters. We conclude that, as a consequence, this framework offers substantial potential advantages. 
}\end{remark}

\section{Discussion and concluding remarks}
This paper has presented a set of results on {\sc mmou}, ranging from procedures to compute moments and a {\sc pde} for the Fourier-Laplace transform, to weak convergence results under specific scalings and a multivariate extension in which multiple {\sc mmou}\,s are modulated by the same background process. Although a relatively large number of aspects is covered, there are many issues that still need to be studied. One such area concerns the large-deviations behavior under specific scalings, so as to obtain the counterparts of the results obtained in e.g.\ \cite{BMT2, JMKO, JM} for the Markov-modulated infinite-server queue. 

It is further remarked that in this paper we looked at an regime-switching version of the {\sc ou} process, but of course we could have considered various other processes. One option is the Markov-modulated version of the so-called Cox-Ingersoll-Ross ({\sc cir}) process:
\[
{\rm d} M \left( t \right) = \left( \alpha_{X \left( t \right)} - \gamma_{X \left( t \right)} M \left( t \right) \right) \, {\rm d} t + \sigma_{X \left( t \right)}\sqrt{M(t)} \, {\rm d} B \left( t \right).
\]
Some results we have established for {\sc mmou} have their immediate {\sc mmcir} counterpart, while for others there are crucial differences. It is relatively straightforward to adapt the  procedure used in Section \ref{FLT}, to set up a system of {\sc pde}\,s for the Fourier-Laplace transforms
(essentially based on It\^{o}'s rule). Interestingly, the recursions to generate all moments are now one-step (rather than two-step) recursions.  A further objective would be to see to what extent the results of our paper generalize to more general classes of diffusions; see e.g.\ \cite{GANG}.

\appendix

\section{Existence and basic properties of MMOU}\label{sec:mmouexist}
In this appendix we provide background and
formal underpinnings of results presented in Section \ref{MOD}.
Throughout we work with a given probability space $\left( \Omega , \mathcal{F} , \mathbb{P} \right)$ on which a random variable $M_0$, a standard Brownian motion $B$, and a continuous-time Markov process $X$ with finite state space are defined. It is assumed that $M_0$, $X$ and $B$ are independent. Denote the natural filtrations of $X$ and $B$ by $( \mathcal{F}_{t}^{X} )_{t \geqslant 0}$  and $( \mathcal{F}_{t}^{B} )_{t \geqslant 0}$, respectively. As before, the state space of $X$ is $\left\lbrace 1 , \ldots , d \right\rbrace$ for some $d\in{\mathbb N}$, and we let $\alpha_{i} \in \mathbb{R}$, $\gamma_{i}>0$ and $\sigma_{i} \in \mathbb{R}$ for $i \in \left\lbrace 1 , \ldots , d \right\rbrace$.  We start by the definition of {\sc mmou}, cf.\ Equation (\ref{SIE:intMMOU}).

\begin{definition}
A stochastic process $M$ is called an {\sc mmou} process with initial condition $M_0$ if
\begin{align}
\label{eq:defmmou}
M \left( t \right) &= M_0 + \int_{0}^{t} \left( \alpha_{X \left( s \right)} - \gamma_{X \left( s \right)} M \left( s \right) \right) \, {\rm d} s + \int_{0}^{t} \sigma_{X \left( s \right)} \, {\rm d} B \left( s \right)
\end{align}
for all $t \geq 0$.
\end{definition}

\vb

To show existence of an {\sc mmou} process, we first need a filtration $\left( \mathcal{H}_t \right)_{t \geqslant 0}$ that satisfies the usual conditions and with respect to which $X$ is adapted and $B$ is a Brownian motion. Define the multivariate process $Y$ by $Y_t = \left( M_0 , X_t , B_t \right)$. Its natural filtration is given by $\mathcal{F}_{t}^{Y} = \sigma \left( M_0 , \mathcal{F}_{t}^{X} , \mathcal{F}_{t}^{B} \right)$. Using the independence assumptions, it is easily verified that $Y$ is a Markov process with respect to $\mathcal{F}_{t}^{Y}$ and that $B$ is a Brownian motion with respect to this filtration. In addition, $Y$ is a Feller process. This is an immediate result of the independence assumptions and the fact that $M_0$ (viewed as a stochastic process), $X$ and $B$ are Feller processes.

Now define the augmented filtration $\left( \mathcal{H}_t \right)_{t \geqslant 0}$ via $\mathcal{H}_t = \sigma \left( \mathcal{F}_{t}^{Y} , \mathcal{N} \right)$, where $\mathcal{N}$ consists of all $F \subset \Omega$ such that there exists $G \in \mathcal{F}_{\infty}^{Y}$ with $F \subset G$ and $\mathbb{P} \left( G \right) = 0$. Since $Y$ has c\`{a}dl\`{a}g paths, it follows from \cite[Prop.
III.2.10]{RY} that $\left( \mathcal{H}_t \right)_{t \geqslant 0}$ satisfies the usual conditions. Relative to this filtration, the process $B$ is a Brownian motion \cite[Th.~2.7.9]{KS} and $Y$ is a Feller process \cite[p.~92]{KS}. 

\vb

We now verify in detail the validity of Thm.\ \ref{PROP}. To construct an {\sc mmou} process, define the stochastic process
\begin{align*}
\Gamma \left( t \right) &:= \int_{0}^{t} \gamma_{X \left( s \right)} \, {\rm d} s,
\end{align*}
which is clearly adapted to $\left( \mathcal{H}_t \right)_{t \geqslant 0}$. The continuous stochastic process
\begin{equation}
\label{eq:gammarepr}
M \left( t \right) = M_0 e^{ - \Gamma \left( t \right)} + \int_{0}^{t} e^{- \left( \Gamma \left( t \right) - \Gamma \left( s \right) \right) }\alpha_{X \left( s \right)} \, {\rm d} s + \int_{0}^{t} e^{- \left( \Gamma \left( t \right) - \Gamma \left( s \right) \right) }\sigma_{X \left( s \right)} \, {\rm d} B \left( s \right)
\end{equation}
is well defined and adapted to $\left( \mathcal{H}_{t} \right)_{t \geqslant 0}$, too. Using similar techniques as in the construction of ordinary {\sc ou} (cf.\ \cite[Ch.~V.5]{RW}), one verifies that $M \left( t \right)$ satisfies Equation~\eqref{eq:defmmou}, so the stochastic process $M$, as given by (\ref{eq:gammarepr}), is an {\sc mmou} process.

\vb

Now we would like to know whether a process that satisfies the stochastic differential equation~\eqref{eq:defmmou} is unique. Of course, uniqueness up to indistinguishability is the strongest form of uniqueness we can get. We will show that this holds for {\sc mmou}.
To this end, suppose we have two {\sc mmou} processes $M^{(1)}$ and $M^{(2)}$, i.e.,
\begin{align*}
M^{(i)}(t) = M_0 + \int_{0}^{t} \left( \alpha_{X \left( s \right)} - \gamma_{X \left( s \right)} M^{(i)} \left( s \right) \right) \, {\rm d} s + \int_{0}^{t} \sigma_{X \left( s \right)} \, {\rm d} B \left( s \right), \qquad i \in \left\lbrace 1,2 \right\rbrace.
\end{align*}
Then $V (t) := M^{(1)} (t) - M^{(2)} (t)$ satisfies
\begin{align*}
V(t) = - \int_{0}^{t} \gamma_{X(s)} V(s) \, {\rm d} s
\end{align*}
with initial condition $V(0) = 0$, on a measurable set $\Omega^{\star}$ that has probability $1$. 
If $V(t) = 0$ for all $t \geq 0$ for every $\omega \in \Omega^{\star}$, then
$M^{(1)}$ and $M^{(2)}$ are indistinguishable.
This is indeed the case, as  a direct consequence of \cite[Th.~I.5.1]{HALE} and \cite[Th.~I.5.3]{HALE}. Consequently, every {\sc mmou} process admits a representation as in Equation~\eqref{eq:gammarepr}.

\vb

For fixed $t \geq 0$ we would like to know the distribution of $M \left( t \right)$. Let, for a given $\Gamma(t)$, $\mu(t)$ and $v(t)$ be given by (\ref{rmu}) and (\ref{rv}), respectively. Observe that we may write \[M \left( t \right) = \mu \left( t \right) + \int_{0}^{t} \exp \left( - \left( \Gamma \left( t \right) - \Gamma \left( s \right) \right) \right) \sigma_{X \left( s \right)} \, {\rm d} B \left( s \right).\]

Using the independence assumptions and standard properties of integrals with respect to Brownian motion, it is easily verified that
\begin{eqnarray*}
{\mathbb E} \left[ e^ { {\rm i} \theta M \left( t \right)} \middle\vert \mathcal{F}_{\infty}^{X} \right] &= &{\mathbb E} \left[ \exp \left( {\rm i} \theta \left( \mu \left( t \right) + \int_{0}^{t}e^{- \left( \Gamma \left( t \right) - \Gamma \left( s \right) \right) } \sigma_{X \left( s \right)} \, {\rm d} B \left( s \right) \right) \right) \middle\vert \mathcal{F}_{\infty}^{X} \right]\\
&=& e^{  {\rm i} \theta \mu \left( t \right) } {\mathbb E} \left[ \exp \left( {\rm i} \theta \left( \int_{0}^{t} e^{- \left( \Gamma \left( t \right) - \Gamma \left( s \right) \right) } \sigma_{X \left( s \right)} \, {\rm d} B \left( s \right) \right) \right) \middle\vert \mathcal{F}_{\infty}^{X} \right]\\
&=& e^{  {\rm i} \theta \mu \left( t \right) }\exp \left( - \tfrac{1}{2} \theta^2 \int_{0}^{t} e^{- 2\left( \Gamma \left( t \right) - \Gamma \left( s \right) \right) } \sigma_{X \left( s \right)}^2 \, {\mathrm d} s \right)\\
&=&
e^{  {\rm i} \theta \mu \left( t \right)- \tfrac{1}{2} \theta^2 v \left( t \right) } ,\end{eqnarray*}which implies the Normality claim in Thm.\ \ref{PROP}.

\end{document}